\RequirePackage{amsthm}
\documentclass[sn-mathphys-num,pdflatex]{sn-jnl}

\usepackage{fix-cm}
\usepackage{graphicx}
\usepackage{multirow}
\usepackage{amsmath,amssymb,amsfonts}
\usepackage{amsthm}
\usepackage{mathrsfs}
\usepackage[title]{appendix}
\usepackage{xcolor}
\usepackage{textcomp}
\usepackage{manyfoot}
\usepackage{booktabs}
\usepackage{listings}
\usepackage{tabularray}
\usepackage{overpic}
\usepackage{mathtools}
\usepackage{enumitem}
\usepackage{hyperref}
\usepackage{tikz}
\usetikzlibrary{positioning}

\theoremstyle{thmstyleone}%
\newtheorem{theorem}{Theorem}
\newtheorem{proposition}[theorem]{Proposition}
\newtheorem{corollary}[theorem]{Corollary}
\newtheorem{lemma}[theorem]{Lemma}

\newtheorem{conjecture}[theorem]{Conjecture}

\theoremstyle{thmstyletwo}%
\newtheorem{example}{Example}
\newtheorem{remark}{Remark}

\theoremstyle{thmstylethree}%
\newtheorem{definition}{Definition}

\renewcommand{\leq}{\leqslant}

\renewcommand{\geq}{\geqslant}

\newcommand{\cB}{\mathcal{B}}

\newcommand{\cD}{\mathcal{D}}

\newcommand{\cH}{\mathcal{H}}

\newcommand{\cK}{\mathcal{K}}

\newcommand{\cM}{\mathcal{M}}

\newcommand{\cP}{\mathcal{P}}

\newcommand{\cT}{\mathcal{T}}

\newcommand{\F}{\mathbb{F}}

\newcommand{\wt}{\operatorname{wt}}
\newcommand{\per}{\operatorname{per}}
\newcommand{\sgn}{\operatorname{sgn}}

\newcommand{\vTwo}{\nu_2}
\newcommand{\whH}{\widehat H}
\newcommand{\eps}{\varepsilon}
\newcommand{\bs}{\textmd{-}}

\newif\ifshowopenproblems
\showopenproblemsfalse

\begin{document}

\title[The Two-Hole Problem]
{Perfect Matchings with Prescribed Differences Beyond Hall: The Two-Hole Problem}

\author[1]{\fnm{Aryeh~Lev} \sur{Zabokritskiy (Yohananov)}}
\email{yuhanalev@telhai.ac.il}

\affil[1]{%
  \orgdiv{Department of Computer Science},
  \orgname{MIGAL -- Galilee Research Institute/Tel-Hai University of Kiryat Shmona and the Galilee},
  \orgaddress{\city{Kiryat Shmona}, \country{Israel}}
}

\abstract{
The Balister--Gy\H{o}ri--Schelp (BGS) conjecture asks whether every zero-sum list of
$2^{s-1}$ nonzero vectors in $\F_2^s$ is the prescribed-difference profile of
a perfect matching.  The conjecture remains open in general, whereas the
classical Hall hyperplane case is solved when all prescribed differences cross
between two affine copies of a hyperplane.  We isolate the smallest mixed case
beyond Hall: exactly two prescribed differences are internal.  Although only
two requests have changed type, the complete Hall permutation is replaced by
a prescribed-difference bijection between two punctured copies of the
hyperplane, with two unknown deleted vertices on each side.  We call this the
\emph{two-hole problem}.

We develop a new combinatorial method for prescribed-difference matchings,
based on counting and the character structure of the binary vector space.
Unlike the known Hall-type methods, which construct a matching through a
sequence of local algorithmic choices, our approach proves existence through
a global noncancellation phenomenon.  This loss of algorithmic structure is
compensated by a different advantage: the method can retain global boundary
information that local exchanges do not control.  As a first application, it
gives a new proof of the binary Hall theorem, and it then yields a complete
solution of the two-hole problem with no multiplicity assumption.  We also
give direct constructive proofs for symmetric even-multiplicity two-hole and
four-hole families.  More broadly, the new technique provides a framework for
studying further subfamilies of the BGS problem by measuring how far their
matching structure departs from the Hall case.

}


\maketitle

\section{Introduction}
\label{sec:intro}

Let $V=\F_2^s$ and put $n=2^{s-1}$.  The prescribed-difference matching
problem asks whether a zero-sum list of $n$ nonzero vectors
$v_0,v_1,\ldots,v_{n-1}\in V$ can always be realized as the edge differences
of a perfect matching of $V$.  Equivalently, one asks for a partition
$V=\bigsqcup_{i=0}^{n-1}\{x_i,y_i\}$ such that $x_i+y_i=v_i$ for every
$i\in[n]$.  The condition $\sum_{i=0}^{n-1}v_i=0$ is necessary because the
sum of all vectors of $V$ is zero.  Balister, Gy\H{o}ri, and Schelp
conjectured that it is also sufficient~\cite{BalisterGyoriSchelp2011}.  We
refer to this statement as the \emph{Balister--Gy\H{o}ri--Schelp conjecture},
or the \emph{BGS conjecture}.

Prescribed-difference pairing problems belong to the broader family of
``seating couples'' questions.  Roughly speaking, the classical cyclic
version asks for analogous pairings of residues modulo an integer rather than
binary vectors.  For cyclic groups of odd prime order, Preissmann and
Mischler proved the corresponding pairing theorem
\cite{PreissmannMischler2009}; Karasev and Petrov developed related polynomial
methods for finite fields and vector spaces~\cite{KarasevPetrov2012}.  The
binary vector-space problem is distinguished by its global zero-sum
obstruction and by the requirement that the matching cover the entire group.

The problem also arises naturally from functional batch codes.  Batch codes
were introduced by Ishai, Kushilevitz, Ostrovsky, and Sahai
\cite{IshaiEtAl2004}.  In the binary functional setting, the nonzero columns
of the Simplex generator matrix are servers storing linear combinations of
$s$ information symbols, and a pair $\{x,y\}$ recovers the request $x+y$.
Thus a prescribed-difference perfect matching of $V$, after removal of a
formal zero server, gives $2^{s-1}$ pairwise disjoint recovery sets.  In
particular, the BGS conjecture implies the functional batch conjecture
$FB(s,2^{s-1})=2^s-1$ of Zhang, Etzion, and
Yaakobi~\cite{ZhangEtzionYaakobi2020}.  Our focus here is the matching
problem itself.

Several structured families are known.  The problem is solved when all
requests are equal and when, after elementary row operations, the request
matrix has an all-one row~\cite{YohananovYaakobi2022}.  The latter condition
is equivalent to saying that all prescribed differences lie outside one fixed
hyperplane.  Taking the parity hyperplane gives the family in which every
prescribed difference has odd Hamming weight.  In particular, the case in
which every prescribed difference is a unit vector is already contained in
this hyperplane family, since every unit vector has odd Hamming weight.  The
same all-crossing family also follows existentially from Hall's theorem
\cite{Hall1952,HollmannEtAl2023}.  The algorithm \emph{BSolution} of
\cite{YohananovYaakobi2022} gives a constructive solution of the same family.
Hall translates the right affine side back to $H$ and works with a permutation
of $H$, whereas \emph{BSolution} keeps the two affine sides explicit and moves
the unresolved redundancy between them through alternating-path exchanges.
This two-sided viewpoint is more convenient for the constructive arguments
later in the paper.  Kov\'acs
proved the conjecture when the number of distinct prescribed differences is at
most $s-2\log_2s-1$, and also, for every fixed $\varepsilon>0$ and all
sufficiently large $s$, when at least a fraction
$\frac12+\varepsilon$ of the prescribed differences are equal
\cite{Kovacs2024}.  A formal summary appears in
Theorem~\ref{thm:known-solvable-families}.

Despite this progress, the full BGS conjecture appears substantially more
difficult.  To expose the source of the difficulty, fix a hyperplane
$H\leq V$ and $\alpha\in V\setminus H$, so that $V=H\sqcup(H+\alpha)$.  A
prescribed difference outside $H$ joins the two affine sides, whereas a
difference in $H$ stays within one side.  We call the former \emph{crossing}
and the latter \emph{internal}.  In this language, a general BGS instance is
an arbitrary mixture of crossing and internal requests.  One must decide where
the internal pairs are placed, which vertices they consume on the two affine
sides, and simultaneously realize all remaining crossing requests between the
unused vertices.  These tasks are strongly coupled: changing one internal
pair changes the two punctured vertex sets on which the crossing matching must
be constructed.

The gap between the functional batch formulation and the prescribed-pair
problem is already visible in an apparently favorable case.  Suppose all
prescribed requests are pairwise distinct.  For the Simplex functional batch
code this causes no difficulty: every request vector is itself a server
column, so the requests can be recovered from distinct singleton servers.  In
the BGS problem, however, every request must be realized by a pair and all
vertices of $V$ must be used.  To our knowledge, even the pairwise-distinct
case remains open in general, apart from structured families such as the Hall
hyperplane case.  Removing all multiplicity phenomena therefore does not
remove the central difficulty of prescribed-difference matching.

The limitations of the existing approaches become especially visible in mixed
instances.  Kov\'acs' method is organized around the span of the distinct
prescribed differences and constructs the matching inside cosets of this
span~\cite{Kovacs2024}.  Its flexibility is therefore tied to the number and
sizes of the available cosets.  When many distinct requests are present, their
span may already be the whole ambient space, leaving no useful coset
structure; a small number of internal requests is not a small parameter for
this method.  The local exchange underlying \emph{BSolution} encounters a
different obstruction.  It preserves previously fixed crossing requests, but
does not control the free boundary left by the exchange.  Once prescribed
internal differences must survive to the end, this missing endpoint control
becomes decisive.  The precise mechanism is explained in
Section~\ref{sec:prelim} and formalized only when it is needed in
Section~\ref{sec:alternating-exchange}.

There is also an exact multivariate polynomial formulation.  After identifying
$V$ with $\F_{2^s}$, one considers
\[
\begin{split}
 f_v(x_1,\ldots,x_n)
 ={}&\prod_{i=1}^{n}v_i
 \prod_{1\leq i<j\leq n}
 (x_i+x_j)(x_i+x_j+v_i)\\
 &\hspace{34mm}\cdot
 (x_i+x_j+v_j)(x_i+x_j+v_i+v_j).
\end{split}
\]
Nonvanishing of $f_v$ at a field point is equivalent to a prescribed pair
decomposition, and reduction modulo
$\langle x_1^{2^s}+x_1,\ldots,x_n^{2^s}+x_n\rangle$ gives an exact
finite-function certificate; see~\cite{YohananovEssayag2025} and the
Combinatorial Nullstellensatz framework of Alon~\cite{Alon1999}.  The most
natural top-degree coefficient, however, disappears in characteristic two.
The highest-degree symmetric part contains
$\prod_{1\leq i<j\leq n}(x_i+x_j)^4$, and the balanced monomial has Dyson
coefficient $\binom{2^s}{2,2,\ldots,2}$, which is even and therefore vanishes
in characteristic two~\cite{Dyson1962}.  After reduction to the finite-function
ring, the remaining coefficients depend simultaneously on all request
variables.  This formulation gives useful equivalent and sufficient
conditions, but it does not presently provide a mechanism adapted to a small
number of internal requests.

The full mixed problem therefore seems beyond the reach of the currently
available tools.  We begin with the smallest possible departure from the
solved Hall hyperplane case.  When every request is crossing, every realizing
edge joins $H$ to $H+\alpha$ and, after translating the right endpoint by
$\alpha$, a perfect matching is exactly a permutation $\pi:H\to H$.  This is
the Hall picture.  The zero-sum condition forces the number of internal
requests to be even, so the first nontrivial mixed case contains exactly two
internal requests.  At first sight this may appear to be only a tiny
perturbation of Hall's theorem.  It is not: already with two internal requests the permutation model is
punctured, the constructive Hall exchange no longer controls the required
boundary, and the known coset methods do not detect the new parameter.

This motivates the basic problem studied in this paper.

\begin{definition}[Two-hole instance]
\label{def:intro-two-hole}
A request matrix is called a \emph{two-hole instance with respect to $H$} if
exactly two requests are internal and all remaining requests are crossing.
\end{definition}

In any realization of a two-hole instance, one internal edge lies in $H$ and
the other in $H+\alpha$.  Hence two vertices are removed from each affine side
before the crossing requests are matched.  After translating the right side
back to $H$, the complete Hall permutation is replaced by an incomplete
permutation, more precisely a prescribed-difference bijection
$H\setminus A\to H\setminus B$ with $|A|=|B|=2$.  Thus there are two holes in
each copy of the Hall space $H$; this is the source of the term
\emph{two-hole}.  The deleted sets $A$ and $B$ are not known in advance and
must be chosen so that the residual crossing instance is realizable.

Figure~\ref{fig:hall-two-hole-intro} summarizes the change from Hall to two
holes.

\begin{figure}[t]
\centering
\resizebox{0.60\textwidth}{!}{%
\begin{tikzpicture}[
    x=0.78cm,
    y=0.76cm,
    vertex/.style={circle,fill=black,inner sep=2pt},
    box/.style={rounded corners=7pt,thick},
    edge/.style={thick,line cap=round}
]
\draw[box] (0,0) rectangle (2.6,5.7);
\draw[box] (5.0,0) rectangle (7.6,5.7);

\foreach \name/\yy in {a1/5.0,a2/4.0,a3/3.0,a4/2.0,a5/1.0}
    \node[vertex] (\name) at (1.3,\yy) {};
\foreach \name/\yy in {b1/5.0,b2/4.0,b3/3.0,b4/2.0,b5/1.0}
    \node[vertex] (\name) at (6.3,\yy) {};

\draw[edge] (a1)--(b2);
\draw[edge] (a2)--(b5);
\draw[edge] (a3)--(b1);
\draw[edge] (a4)--(b4);
\draw[edge] (a5)--(b3);

\node[draw=none,fill=none] at (1.3,-0.45) {$H$};
\node[draw=none,fill=none] at (6.3,-0.45) {$H+\alpha$};
\node[draw=none,fill=none,font=\small] at (3.8,6.15)
{\textup{(a) Hall / \emph{BSolution}}};

\begin{scope}[xshift=9.2cm]
\draw[box] (0,0) rectangle (2.6,5.7);
\draw[box] (5.0,0) rectangle (7.6,5.7);

\foreach \name/\yy in {c1/5.05,c2/4.05,c3/3.05,c4/1.25,c5/0.72}
    \node[vertex] (\name) at (1.3,\yy) {};
\foreach \name/\yy in {d1/5.10,d2/4.10,d3/3.10,d4/1.85,d5/0.48}
    \node[vertex] (\name) at (6.3,\yy) {};

\draw[edge] (c1)--(d2);
\draw[edge] (c2)--(d3);
\draw[edge] (c3)--(d1);
\draw[edge] (c4)--(c5);
\draw[edge] (d4)--(d5);

\node[draw=none,fill=none] at (1.3,-0.45) {$H$};
\node[draw=none,fill=none] at (6.3,-0.45) {$H+\alpha$};
\node[draw=none,fill=none,font=\small] at (3.8,6.15)
{\textup{(b) Two holes}};
\end{scope}
\end{tikzpicture}%
}
\caption{The all-crossing Hall / \emph{BSolution} case and the first mixed
case.  Hall encodes the left panel as a permutation of $H$, whereas
\emph{BSolution} keeps the two affine sides explicit.  With two internal
requests, two vertices are deleted from each side and the Hall permutation is
replaced by a bijection between two copies of $H$ with two holes each.}
\label{fig:hall-two-hole-intro}
\end{figure}
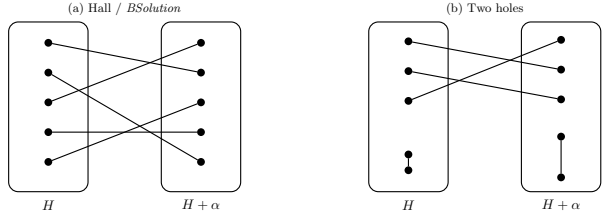

The main idea of this paper is a character-minor method that sees this
punctured Hall structure directly.  The method first yields a new
character-theoretic proof of the binary Hall theorem.  Attach a variable
$z_x$ to each reduced difference $x\in H$ and form the group circulant
$C(z)=(z_{a+b})_{a,b\in H}$.  A coefficient of $\per C(z)$ counts Hall
realizations exactly, while the same coefficient of $\det C(z)$ is their
signed count.  For every zero-sum multiplicity profile $\mathbf m$, we prove
\[
    \vTwo\!\left([z^{\mathbf m}]\det C(z)\right)
    =
    \vTwo\!\left(\frac{n!}{\mathbf m!}\right),
\]
where $\vTwo(a)$ denotes the exponent of $2$ in the prime factorization of a
nonzero integer $a$.  Thus complete cancellation between even and odd Hall
realizations is impossible.  Walsh factorization converts the coefficient
into the permanent of a repeated-column character matrix, and pairing rows by
a character involution gives an exact recursion from order $n$ to order
$n/2$.

Our second and main contribution is the complete two-hole theorem.  The part
of Hall that survives is the crossing core: after the two internal edges are
reserved, the remaining requests still ask for a prescribed-difference
bijection between two affine copies of $H$.  What is lost is the full domain
and codomain.  The crossing core now lives on two unknown punctured copies of
$H$, so the full circulant must be replaced by a complementary minor.  We
combine Cauchy--Binet, Jacobi's complementary-minor identity, and exact
$2$-adic nonvanishing for two deleted Walsh factors.  An invertible Walsh
transform then forces a choice of the two internal-edge locations for which
the deleted circulant coefficient is nonzero.  This proves every zero-sum
two-hole instance with no multiplicity assumption.

The two-hole theorem is existential.  We also give constructive quotient
proofs for two symmetric even-multiplicity families: the two-hole case and the
four-hole patterns $p,p,p,p$ and $p,p,q,q$.  In each quotient, the requests to
be processed form an all-crossing Hall instance on the full quotient ambient
space, so \emph{BSolution} applies constructively.  The $p,p,q,q$ case requires
one additional idea: the coset-level procedure is stopped before its last
step, because that step would destroy a protected coset distance, and the
argument finishes with one final point-level \emph{BSolution} step and the global sum
condition.

\ifshowopenproblems
\else
The present paper deliberately focuses on the first mixed case beyond Hall.
Its aim is to isolate the structural change created by two internal requests,
to identify the exact boundary information lost by the complete Hall
permutation, and to develop a method that restores this information.  The main
result is therefore stated and proved at the two-hole level, without imposing
a multiplicity assumption.  To our knowledge, the two-hole problem has not
previously been isolated as a separate prescribed-difference matching problem.

More importantly, the proof produces a mechanism absent from the known Hall,
coset, and \emph{BSolution} approaches: complementary character minors retain
the exact boundary data created by the two internal requests.  This is the
structural reason the method can pass from the complete Hall permutation to
the first punctured case.
\fi

The paper is organized as follows.  Section~\ref{sec:prelim} fixes notation
and formalizes the hyperplane viewpoint.  Section~\ref{sec:hall} develops the
group-circulant and Walsh-permanent method and proves exact Hall
noncancellation.  Section~\ref{sec:two-hole} passes to complementary minors and
proves the two-hole theorem.  Section~\ref{sec:constructive-complements}
returns to the alternating-path exchange behind \emph{BSolution}, explains its exact
endpoint limitation beyond Hall, and gives the symmetric multiplicity-based
constructions.  \ifshowopenproblems
\else
We conclude by summarizing the two-hole viewpoint and the role of the
character-minor method.
\fi

\section{Preliminaries and the Hyperplane Viewpoint}
\label{sec:prelim}

Throughout the paper all vectors and matrices are over $\F_2$.  For a positive
integer $n$, write $[n]=\{0,1,\ldots,n-1\}$, and for a binary vector $x$ let
$\wt(x)$ denote its Hamming weight.  Fix $s\geq2$ and put $n=2^{s-1}$.  A
\emph{request} is a nonzero vector of $\F_2^s$, and a request instance is an
$s\times n$ matrix $M=[v_0,v_1,\ldots,v_{n-1}]$ whose columns need not be
distinct.

Let $\cH=[h_0,h_1,\ldots,h_{2n-1}]$ be an $s\times2^s$ Hadamard generator
matrix, so its columns are precisely the vectors of $\F_2^s$, each appearing
once.  We regard these columns as servers.  The zero column is retained for the
prescribed-difference formulation and removed when passing to functional batch
codes.

\begin{definition}[Hadamard pair solution]
\label{def:hadamard-solution}
The request matrix $M$ has a \emph{Hadamard pair solution} if the columns of
$\cH$ can be ordered so that $v_i=h_{2i}+h_{2i+1}$ for every $i\in[n]$.
Equivalently, $\F_2^s$ can be partitioned into $n$ pairs whose prescribed
differences are the columns of $M$.
\end{definition}

A Hadamard pair solution necessarily satisfies
$\sum_{i=0}^{n-1}v_i=\sum_{x\in\F_2^s}x=0$.  The BGS conjecture asserts that
this condition is sufficient.

\begin{conjecture}[Balister--Gy\H{o}ri--Schelp
\cite{BalisterGyoriSchelp2011}]
\label{conj:bgs}
Let $M=[v_0,\ldots,v_{2^{s-1}-1}]$ have nonzero columns and satisfy
$\sum_{i=0}^{2^{s-1}-1}v_i=0$.  Then $M$ has a Hadamard pair solution.
\end{conjecture}

An $FB\bs(N,s,k)$ functional batch code consists of $N$ binary servers storing
linear combinations of $s$ information symbols such that every multiset of
$k$ linear requests has $k$ pairwise disjoint recovery sets.  We write
$FB(s,k)$ for the minimum possible number of servers.  The conjecture of
\cite{ZhangEtzionYaakobi2020} states that $FB(s,2^{s-1})=2^s-1$.
Conjecture~\ref{conj:bgs} implies this functional batch conjecture through the
virtual-zero reduction recalled in the introduction.

We shall repeatedly change coordinates, so we record the following invariance.

\begin{lemma}[Linear invariance]
\label{lem:linear-invariance}
Let $T\in\operatorname{GL}(s,2)$.  The request matrix
$M=[v_0,\ldots,v_{n-1}]$ has a Hadamard pair solution if and only if
$TM=[Tv_0,\ldots,Tv_{n-1}]$ has a Hadamard pair solution.
\end{lemma}

\begin{proof}
Suppose $v_i=h_{2i}+h_{2i+1}$ for every $i\in[n]$.  Since $T$ is invertible,
$\{Tx:x\in\F_2^s\}=\F_2^s$, so $T\cH$ is only a reordering of a Hadamard
generator matrix and $Tv_i=Th_{2i}+Th_{2i+1}$.  The converse follows by
applying $T^{-1}$.
\end{proof}

We now formalize the hyperplane picture used throughout the paper.  Fix a
hyperplane $H\leq\F_2^s$ and choose $\alpha\notin H$, so that
$\F_2^s=H\sqcup(H+\alpha)$.  A request $v\notin H$ is \emph{crossing with
respect to $H$}, while a request $v\in H$ is \emph{internal with respect to
$H$}.  Every edge realizing a crossing difference joins the two affine sides;
every edge realizing an internal difference stays inside one side.

Following Definition~\ref{def:intro-two-hole}, we call an internal request a
\emph{hole with respect to $H$} and count holes with multiplicity.  More
generally, an \emph{$h$-hole instance with respect to $H$} has exactly $h$
internal requests and all remaining requests are crossing.  The word
\emph{hole} refers to the Hall matching, not to a missing request: an internal
edge uses two vertices on one affine side, so those vertices are removed from
the crossing matching between the two copies of $H$.

There is a useful numerical consequence of this picture.  In any solution of
an $h$-hole instance, let $a$ internal edges lie in $H$ and $b$ lie in
$H+\alpha$.  Since the $n-h$ crossing edges use exactly $n-h$ vertices on each
affine side, precisely $h$ vertices remain on each side for internal edges.
Hence $2a=2b=h$.  Thus $h$ is even, exactly $h/2$ internal edges lie on each
side, and after the internal edges are removed the residual problem is a
prescribed-difference bijection between two copies of $H$ with exactly $h$
vertices deleted from each.  For $h=0$ this is the complete Hall permutation.
For $h=2$ it is an incomplete Hall permutation
$H\setminus A\to H\setminus B$ with two missing vertices in each copy of
$H$---the two holes.

The parity hyperplane
$H_{\rm even}=\{x\in\F_2^s:\wt(x)\equiv0\pmod2\}$ gives the most concrete
example.  Internal requests are exactly the even-Hamming-weight requests and
crossing requests are the odd-weight requests.  Consequently, a two-hole
instance with respect to $H_{\rm even}$ is precisely an instance with two
even-weight requests and $n-2$ odd-weight requests.

The relation with \emph{BSolution} is worth making explicit already at the
geometric level.  The algorithm acts on the full space
$V=H\sqcup(H+\alpha)$, not on $H$ alone, and in the all-crossing case it is a
constructive counterpart of Hall.  A \emph{BSolution} step fixes one crossing
request while transferring the unresolved redundancy through the two affine
sides and preserving all previously fixed crossing requests.  What it does not
control is the boundary left by the transfer: neither the terminal pair on the
opposite side nor the vertex released on the starting side is prescribed.

This becomes decisive in the two-hole problem.  \emph{BSolution} can realize
crossing requests one after another until only three requests remain
unresolved.  At that moment exactly three free vertices remain in $H$ and
three in $H+\alpha$; the unresolved requests are the last crossing request and
the two prescribed internal requests $p$ and $q$.  Suppose two free vertices
on one side have been reserved at difference $p$.  Only one other free vertex
remains on that side, so the next \emph{BSolution} step must use at least one
vertex of the reserved pair.  The exchange replaces it in the free boundary by
an uncontrolled released vertex, and the preserved difference $p$ is lost.
The new free pair may accidentally have difference $p$ or $q$, but no invariant
forces either outcome.  Thus the constructive Hall mechanism gets almost to
the end and fails precisely because it does not control the final boundary.
This is the obstruction that motivates the deleted-minor method.  The graph
$G_x(\cH)$ and the formal alternating-path switch are postponed until
Section~\ref{sec:alternating-exchange}, immediately before their constructive
use.

We finish the preliminaries by recording the main previously known solvable
families.

\begin{theorem}[Known solvable families]
\label{thm:known-solvable-families}
Let $M=[v_0,\ldots,v_{n-1}]$ be a zero-sum request matrix with nonzero
columns.  A Hadamard pair solution is known to exist in each of the following
cases.
\begin{enumerate}[label=\textup{(\roman*)}]
    \item The ambient dimension satisfies $s\leq5$
    \cite{BalisterGyoriSchelp2011}.

    \item All requests are equal~\cite{YohananovYaakobi2022}.

    \item There exists a nonzero linear functional
    $\ell:\F_2^s\to\F_2$ such that $\ell(v_i)=1$ for every $i$.
    Equivalently, after elementary row operations the request matrix has an
    all-one row, or all requests lie outside one fixed hyperplane.  This family
    has a constructive solution by \emph{BSolution}
    \cite{YohananovYaakobi2022} and an existential Hall-theorem proof
    \cite{Hall1952,HollmannEtAl2023}.  Taking $\ell$ to be the parity
    functional gives every all-odd-weight instance, and therefore every
    unit-vector instance.

    \item After reordering the columns, there is a nonzero $u$ such that
    $v_0=\cdots=v_{n/2-1}=u$ and the remaining $n/2$ requests can be
    partitioned into equal-valued pairs.  This is the
    Balister--Gy\H{o}ri--Schelp majority family
    \cite{BalisterGyoriSchelp2011}.

    \item The number $t$ of distinct prescribed differences satisfies
    $t\leq s-2\log_2s-1$~\cite{Kovacs2024}.

    \item For every fixed $\varepsilon>0$ and all sufficiently large $s$, at
    least a fraction $\frac12+\varepsilon$ of the prescribed differences are
    equal~\cite{Kovacs2024}.
\end{enumerate}
\end{theorem}

Theorem~\ref{thm:known-solvable-families} highlights the hyperplane gap.  An
all-one row is exactly the Hall case $h=0$, whereas an all-zero row places all
requests inside one hyperplane and remains open in general.  The main theorem
of this paper settles the first intermediate value $h=2$.

\section{Hall Noncancellation and Character Permanents}
\label{sec:hall}

Fix the hyperplane $H$ and $\alpha\notin H$ from
Section~\ref{sec:prelim}.  Put $r=s-1$ and $n=|H|=2^r$, and identify $H$
with $\F_2^r$ after choosing a basis.

In the Hall case all requests lie in $H+\alpha$.  Write the original
request matrix as $M=[v_0,v_1,\ldots,v_{n-1}]$ and each request uniquely as
$v_i=\alpha+d_i$ with $d_i\in H$, equivalently $d_i=v_i+\alpha$.  Under the
affine identification $\alpha+x\mapsto x$ from $H+\alpha$ to $H$, the request
$v_i$ is represented by the reduced difference $d_i$.  We collect these values
in $\cD=[d_0,d_1,\ldots,d_{n-1}]$.  Zero is allowed among the $d_i$.

A realization of $\cD$ is a permutation $\pi:H\to H$ for which the
multiset $\{a+\pi(a):a\in H\}$ is precisely the multiset of columns of
$\cD$.  This is exactly the perfect cross-hyperplane matching problem for
$M$: the pair $\{a,\alpha+\pi(a)\}$ has difference
$\alpha+(a+\pi(a))$, so $a+\pi(a)=d_i$ realizes the original request
$\alpha+d_i=v_i$.  Conversely, every perfect matching between $H$ and
$H+\alpha$ determines such a permutation after translating the second
endpoint by $\alpha$.

The following existence statement is already known.  It is the elementary
abelian $2$-group case of Hall's classical theorem~\cite{Hall1952}.  In the
prescribed-difference setting, the same all-crossing case also follows from the
constructive \emph{BSolution} method of \cite{YohananovYaakobi2022}; see also
\cite{HollmannEtAl2023}.  We use Hall's reduced formulation here because a
permutation of $H$ is the natural input for the group circulant.  For the later
constructive arguments, \emph{BSolution} has a different advantage: it keeps
both affine sides visible and exposes the local motion of the redundancy.

Our purpose here is different.  We give another proof based on coefficients of
matrix polynomials.  The resulting viewpoint contains more enumerative and
arithmetic information than existence alone: one coefficient counts the exact
number of realizations of a prescribed difference multiset, while another gives
their signed count.  We shall prove that, for every zero-sum prescribed-
difference matrix, this signed count has an exact $2$-adic valuation and in
particular cannot vanish.  Thus the Binary Hall theorem will follow again from
a stronger noncancellation statement.  The same coefficient framework will
also admit a minor version in the next section.

\subsection{The group circulant and the signed realization count}
\label{subsec:group-circulant}

\begin{theorem}[Binary Hall theorem]
\label{thm:binary-hall}
Let
\[
    \cD=[d_0,d_1,\ldots,d_{n-1}],
    \qquad
    d_i\in H,
\]
and suppose
\[
    \sum_{i=0}^{n-1}d_i=0.
\]
Then $\cD$ has a realization.
\end{theorem}

We first explain why matrix polynomials are relevant to the realization
problem.  Attach a variable $z_x$ to each possible prescribed difference
$x\in H$ and define the group circulant
\[
    C(z)=\bigl(z_{a+b}\bigr)_{a,b\in H}.
\]
The rows and columns of $C(z)$ are indexed by possible endpoints
$a,b\in H$, and the $(a,b)$ entry is labeled by the difference $a+b\in H$.
Thus $a$ and $b$ are endpoints, whereas the column variable records a possible
prescribed difference.  Fix an ordering of $H$ and let $S_n$ be the symmetric
group on $[n]$.  Through this ordering, each $\pi\in S_n$ is also a bijection
$\pi:H\to H$, and $\sgn(\pi)\in\{+1,-1\}$ denotes its sign.

Choosing one entry from every row and every column of $C(z)$ is exactly the
same as choosing a permutation $\pi\in S_n$.  The corresponding product is
\[
    \prod_{a\in H}z_{a+\pi(a)}.
\]
Its variables record precisely the multiset of prescribed differences realized
by $\pi$.  Notice that the values $a+\pi(a)$ need not be distinct: the
permutation condition prevents collisions between endpoints, not repetitions
among the realized differences.

Consequently,
\[
    \per C(z)
    =
    \sum_{\pi\in S_n}
    \prod_{a\in H}z_{a+\pi(a)},
\]
whereas
\[
    \det C(z)
    =
    \sum_{\pi\in S_n}
    \sgn(\pi)
    \prod_{a\in H}z_{a+\pi(a)}.
\]

We now encode the fixed prescribed-difference matrix $\cD$.  For each
$x\in H$, let
\[
    m_x=|\{i\in[n]:d_i=x\}|,
    \qquad
    \mathbf m=(m_x)_{x\in H}.
\]
Thus $m_x$ is exactly the number of occurrences of the prescribed difference
$x$ among the columns of $\cD$.  We write
\[
    |\mathbf m|=\sum_{x\in H}m_x,
    \qquad
    \sigma(\mathbf m)=\sum_{x\in H}m_xx\in H,
\]
and
\[
    z^{\mathbf m}=\prod_{x\in H}z_x^{m_x},
    \qquad
    \mathbf m!=\prod_{x\in H}m_x!.
\]
Since $\cD$ has $n$ columns, $|\mathbf m|=n$ and
$\sigma(\mathbf m)=\sum_{i=0}^{n-1}d_i$.  The monomial $z^{\mathbf m}$
therefore records exactly the multiplicities of the prescribed differences in
$\cD$.

The permanent and determinant coefficients now have direct interpretations.
The coefficient $[z^{\mathbf m}]\per C(z)$ is exactly the number of
realizations of $\cD$: every permutation with the prescribed difference
multiset contributes $1$, and no other permutation contributes.  By contrast,
$[z^{\mathbf m}]\det C(z)$ is the signed number of realizations:
\[
    [z^{\mathbf m}]\det C(z)
    =
    \#\{\text{even realizations of }\cD\}
    -
    \#\{\text{odd realizations of }\cD\}.
\]
There is an immediate support restriction on the permanent coefficient.  For
any permutation $\pi:H\to H$,
\[
    \sum_{a\in H}(a+\pi(a))
    =
    \sum_{a\in H}a+\sum_{b\in H}b
    =0.
\]
Hence
\[
    \sigma(\mathbf m)\neq0
    \quad\Longrightarrow\quad
    [z^{\mathbf m}]\per C(z)=0.
\]
Combinatorially, this is simply the necessary zero-sum condition: a perfect
matching uses every server exactly once, so the sum of its realized requests is
the sum of all server vectors.  Later we shall recover the corresponding
vanishing on the determinant side from the character structure itself.

\begin{example}[From crossing requests to a Hall coefficient]
\label{ex:running-hall}
Let
\[
    V=\F_2^3,
    \qquad
    H=\{0\}\times\F_2^2,
    \qquad
    \alpha=100.
\]
Consider the all-crossing request list
\[
    M=[101,101,110,110].
\]
The requests are zero-sum.  Translating the right affine side
$H+\alpha$ back to $H$ replaces each crossing request $v_i$ by the reduced
difference
\[
    d_i=v_i+\alpha.
\]
After identifying $H$ with its last two coordinates, we obtain
\[
    \cD=[01,01,10,10].
\]

There are exactly two perfect crossing matchings of $V$ with this
prescribed-difference profile:
\[
\begin{aligned}
\cM_1={}&
\bigl\{
\{000,101\},
\{001,111\},
\{010,100\},
\{011,110\}
\bigr\},\\
\cM_2={}&
\bigl\{
\{000,110\},
\{001,100\},
\{010,111\},
\{011,101\}
\bigr\}.
\end{aligned}
\]
Indeed, the edge differences of either matching are two copies of $101$ and
two copies of $110$.

After translating the right endpoints by $\alpha$, the two matchings become
the two permutations of $H$
\[
    (01,11,00,10)
    \qquad\text{and}\qquad
    (10,00,11,01)
\]
in one-line notation with respect to the ordering $00,01,10,11$.
The group circulant is
\[
C(z)=
\begin{pmatrix}
 z_{00}&z_{01}&z_{10}&z_{11}\\
 z_{01}&z_{00}&z_{11}&z_{10}\\
 z_{10}&z_{11}&z_{00}&z_{01}\\
 z_{11}&z_{10}&z_{01}&z_{00}
\end{pmatrix}.
\]
Each of the two permutations contributes the monomial
\[
    z_{01}^2z_{10}^2
\]
to $\per C(z)$, and no other permutation contributes this monomial.  Hence
\[
    [z_{01}^2z_{10}^2]\per C(z)=2,
\]
exactly the number of Hall realizations of the original request list $M$.
Both permutations are odd, so
\[
    [z_{01}^2z_{10}^2]\det C(z)=-2.
\]
Thus the permanent coefficient literally counts the crossing matchings of the
original Hall instance, while the determinant coefficient counts the same
matchings with signs.  We shall return to this profile several times to
illustrate the character compression.
\end{example}

Under the zero-sum assumption, Hall's theorem and \emph{BSolution} already
imply $[z^{\mathbf m}]\per C(z)>0$.  They do not, however, rule out complete
cancellation between even and odd realizations in the determinant coefficient.
The determinant method below proves the stronger fact that such complete
cancellation cannot occur.

\begin{lemma}[Circulant coefficient certificate]
\label{lem:circulant-certificate}
Let $|\mathbf m|=n$.  If
\[
    [z^{\mathbf m}]\det C(z)\neq0,
\]
then the corresponding prescribed-difference matrix $\cD$ has a realization.
\end{lemma}

\begin{proof}
Expanding the determinant gives
\[
    \det C(z)
    =
    \sum_{\pi\in S_n}
    \sgn(\pi)
    \prod_{a\in H}z_{a+\pi(a)}.
\]
Hence
\[
    [z^{\mathbf m}]\det C(z)
\]
is the signed sum of exactly those permutations whose realized difference
multiplicities are $\mathbf m$.  If this signed sum is nonzero, then at least
one such permutation exists.  That permutation is a realization of $\cD$.
\end{proof}

Lemma~\ref{lem:circulant-certificate} is only a sufficient certificate:
existence of a realization does not imply that the signed count is nonzero,
since even and odd realizations could cancel.  By contrast,
$[z^{\mathbf m}]\per C(z)>0$ is equivalent to existence and gives the exact
number of realizations.  All determinant and permanent coefficients below are
integers.

For an integer $a\neq0$, the \emph{$2$-adic valuation} $\vTwo(a)$ is the
largest integer $t\geq0$ such that $2^t$ divides $a$; set
$\vTwo(0)=+\infty$.  Thus $\vTwo(a)=t<\infty$ says simultaneously that
$a\neq0$, $2^t\mid a$, and $2^{t+1}\nmid a$.

The stronger noncancellation statement that drives our proof is the following.

\begin{theorem}[Full coefficient nonvanishing]
\label{thm:full-coefficient}
Let $|\mathbf m|=n$.  Then
\begin{enumerate}[label=\textup{(\roman*)}]
    \item if $\sigma(\mathbf m)\neq0$, then
    \[
        [z^{\mathbf m}]\det C(z)=0;
    \]
    \item if $\sigma(\mathbf m)=0$, then
    \[
        \vTwo\!\left([z^{\mathbf m}]\det C(z)\right)
        =
        \vTwo\!\left(\frac{n!}{\mathbf m!}\right),
    \]
    and in particular
    \[
        [z^{\mathbf m}]\det C(z)\neq0.
    \]
\end{enumerate}
\end{theorem}

Part~\textup{(ii)} is strictly stronger than the existence conclusion of
Theorem~\ref{thm:binary-hall}.  Hall's theorem and \emph{BSolution} guarantee at least
one realization, equivalently a positive permanent coefficient.  Theorem~\ref{thm:full-coefficient} determines the exact $2$-adic valuation of the
signed determinant coefficient and therefore proves that complete cancellation
between even and odd realizations is impossible.

The matrix $C(z)$ directly encodes the prescribed-difference problem, but its
coefficients are difficult to analyze.  We therefore introduce a closely
related sign matrix with considerably more structure.  The new matrix
diagonalizes $C(z)$, so their determinants are connected by an explicit
identity.  After extracting the coefficient selected by the multiplicity
vector $\mathbf m$, the problem becomes a permanent calculation for a
repeated-column version of the new matrix.  We shall then carry out the
noncancellation argument entirely in this new setting and return to $C(z)$ only
at the end.

\subsection{Walsh factorization and repeated-column permanents}

Identify the character group $\whH$ with $H$ and put
$\eps_\rho(x)=(-1)^{\langle\rho,x\rangle}$ for $\rho,x\in H$.  Define the
character matrix
\[
    W=\bigl(\eps_\rho(x)\bigr)_{\rho\in\whH,\,x\in H}.
\]
Its rows are indexed by characters and its columns by possible prescribed
differences $x\in H$.  Thus a column index is a possible value of $a+b$, not
an endpoint.

In the binary setting $W$ is the Walsh--Hadamard character matrix; the Walsh
system originates in \cite{Walsh1923}.  We use the name only for this character
transform and the matrices derived from it.

For each $\rho\in\whH$, define the linear form
\[
    L_\rho(z)=\sum_{x\in H}\eps_\rho(x)z_x.
\]
Character orthogonality gives
\[
    WW^{\mathsf T}=nI_n,
    \qquad
    W^{-1}=n^{-1}W^{\mathsf T}.
\]

The character linear forms may be arranged as the diagonal entries of the
matrix
\[
    \Lambda(z)
    =
    \begin{pmatrix}
        L_{\rho_0}(z) & 0 & \cdots & 0\\
        0 & L_{\rho_1}(z) & \cdots & 0\\
        \vdots & \vdots & \ddots & \vdots\\
        0 & 0 & \cdots & L_{\rho_{n-1}}(z)
    \end{pmatrix},
\]
where
\[
    \whH=\{\rho_0,\rho_1,\ldots,\rho_{n-1}\}.
\]
Thus $\Lambda(z)$ is simply the matrix whose diagonal entries are the forms
$L_\rho(z)$.

The following diagonalization is standard for group circulants.  We include the
short calculation in our notation.

\begin{lemma}[Character diagonalization of the group circulant]
\label{lem:character-diagonalization}
With $\Lambda(z)$ as above,
\[
    C(z)=\frac1n W^{\mathsf T}\Lambda(z)W.
\]
Consequently,
\[
    \det C(z)=\prod_{\rho\in\whH}L_\rho(z).
\]
\end{lemma}

\begin{proof}[Proof of Lemma~\ref{lem:character-diagonalization}]
The $(a,b)$ entry of the right-hand side is
\begin{align*}
\frac1n\sum_{\rho\in\whH}
\eps_\rho(a)L_\rho(z)\eps_\rho(b)
&=
\frac1n\sum_{x\in H}z_x
\sum_{\rho\in\whH}\eps_\rho(a+b+x)\\
&=z_{a+b},
\end{align*}
by character orthogonality.  Hence
\[
    C(z)=\frac1nW^{\mathsf T}\Lambda(z)W.
\]
Taking determinants and using
\[
    \det(W)^2=n^n
\]
gives
\[
    \det C(z)=\prod_{\rho\in\whH}L_\rho(z).
\]
\end{proof}

Put
\[
    P_H(z)=\prod_{\rho\in\whH}L_\rho(z),
    \qquad
    b(\mathbf m)=[z^{\mathbf m}]P_H(z).
\]
By Lemma~\ref{lem:character-diagonalization},
\[
    b(\mathbf m)=[z^{\mathbf m}]\det C(z).
\]
Thus the coefficient in Theorem~\ref{thm:full-coefficient} may now be studied
inside the product of character linear forms.

For the fixed prescribed-difference matrix
\[
    \cD=[d_0,d_1,\ldots,d_{n-1}],
\]
define the \emph{prescribed-difference character matrix}
\[
    W_{\cD}
    =
    \bigl(\eps_\rho(d_i)\bigr)_{\rho\in\whH,\,i\in[n]}.
\]
The distinction between $W$ and $W_{\cD}$ is important.  The columns of $W$
are indexed by all possible difference values $x\in H$, with one column for
each $x$.  By contrast, the columns of $W_{\cD}$ are indexed by the labeled
occurrences
\[
    d_0,d_1,\ldots,d_{n-1}
\]
in the fixed prescribed-difference matrix $\cD$.

More precisely, the $i$th column of $W_{\cD}$ is the column of $W$ indexed by
the difference value $d_i$.  Thus equal prescribed differences give equal but
\emph{labeled} columns.  Equivalently, if $m_x$ is the multiplicity of $x$ in
$\cD$, then $W_{\cD}$ is obtained from $W$ by taking $m_x$ labeled copies of
the column indexed by $x$, for every $x\in H$.  Hence $W_{\cD}$ depends on
$\cD$ only through $\mathbf m$, up to a permutation of its columns.

For an $n\times n$ matrix $A=(a_{ij})$, its \emph{permanent} is
\[
    \per A
    =
    \sum_{\pi\in S_n}
    \prod_{i=0}^{n-1}a_{i,\pi(i)}.
\]
It has the same expansion as a determinant but without the factor
$\sgn(\pi)$.  A permanent can in general be zero.

After choosing a basis of $H$ and ordering its elements accordingly, the
character matrix $W$ is the classical Sylvester--Hadamard matrix of order
$n=2^r$, up to permutations of rows and columns; see the recursive Sylvester
construction \cite{Sylvester1867}.  Since the permanent is invariant under row
and column permutations, known results for the Sylvester--Hadamard matrix apply
directly to $W$.

In particular, Chabaud proved that, for $r\geq2$,
\[
    \vTwo(\per W)=\vTwo(n!)=n-1,
\]
and hence
\[
    \per W\neq0
\]
\cite{Chabaud2018}.

This result, however, treats only the special prescribed-difference profile
\[
    m_x=1
    \qquad
    \text{for every }x\in H.
\]
Equivalently, it treats the special case in which the columns of the fixed
matrix $\cD$ are a permutation of the elements of $H$.  In that case
$W_{\cD}$ is equal to $W$ up to a permutation of its columns, and the relevant
coefficient is the square-free coefficient
\[
    \left[\prod_{x\in H}z_x\right]P_H(z).
\]
This is a condition on the prescribed-difference matrix $\cD$.  It is not a
consequence of the realizing permutation $\pi$: although $\pi$ uses every
endpoint exactly once, the values $a+\pi(a)$ may repeat.

A general Hall instance may repeat some prescribed differences many times and
omit other values of $H$ entirely.  The running profile from
Example~\ref{ex:running-hall},
\[
    \cD=[01,01,10,10],
\]
requires the coefficient
\[
    [z_{01}^2z_{10}^2]P_H(z)
\]
and the repeated-column matrix $W_{\cD}$, rather than $W$.  Thus Chabaud's
nonvanishing theorem for $\per W$ does not prove the Hall statement needed
here.  We require a valuation theorem for $\per W_{\cD}$ for every zero-sum
prescribed-difference matrix $\cD$.

To exploit the character rows in the valuation argument, we want to replace the coefficient
of the product of linear forms by a permanent whose columns are the labeled
requests themselves.  Repeated difference values then appear as repeated but
distinct labeled columns.  The only discrepancy is the number of ways to
label equal occurrences, which is exactly the factorial factor below.

\begin{lemma}[Coefficient--permanent correspondence]
\label{lem:coeff-permanent}
For every prescribed-difference matrix $\cD$ with multiplicity vector
$\mathbf m$,
\[
    \mathbf m!\,b(\mathbf m)=\per W_{\cD}.
\]
\end{lemma}

The correspondence is easiest to see directly in the matrices.

\begin{example}[Labeled repetitions in the running Hall profile]
\label{ex:character-summary}
Return to the reduced profile
\[
    \cD=[01,01,10,10]
\]
from Example~\ref{ex:running-hall}.  Before repeating columns, attach the
variable $z_x$ to the Walsh column indexed by the possible difference
$x\in H$:
\[
\widetilde W(z)
:=
W
\begin{pmatrix}
z_{00}&0&0&0\\
0&z_{01}&0&0\\
0&0&z_{10}&0\\
0&0&0&z_{11}
\end{pmatrix}
=
\begin{array}{c|cccc}
 &00&01&10&11\\ \hline
00& z_{00}& z_{01}& z_{10}& z_{11}\\
01& z_{00}&-z_{01}& z_{10}&-z_{11}\\
10& z_{00}& z_{01}&-z_{10}&-z_{11}\\
11& z_{00}&-z_{01}&-z_{10}& z_{11}
\end{array}.
\]
Here there is one column for each \emph{possible difference value}.  The sum
of row $\rho$ is exactly the character form $L_\rho(z)$.

The prescribed list $\cD$ does not use the four possible values once each.
It asks for two labeled occurrences of $01$ and two labeled occurrences of
$10$.  Repeating the corresponding columns gives the variable-decorated
matrix
\[
\widetilde W_{\cD}(z)
=
\begin{array}{c|cccc}
 &d_0=01&d_1=01&d_2=10&d_3=10\\ \hline
00& z_{01}& z_{01}& z_{10}& z_{10}\\
01&-z_{01}&-z_{01}& z_{10}& z_{10}\\
10& z_{01}& z_{01}&-z_{10}&-z_{10}\\
11&-z_{01}&-z_{01}&-z_{10}&-z_{10}
\end{array}.
\]
Removing the displayed variables leaves the prescribed-difference character
matrix $W_{\cD}$.  The two $z_{01}$ columns and the two $z_{10}$ columns make
the source of the factorial factor visible: the polynomial coefficient sees
only two occurrences of each value, whereas the permanent sees four labeled
columns.

Here
\[
    m_{01}=m_{10}=2,
    \qquad
    \mathbf m!=2!2!=4.
\]
Example~\ref{ex:running-hall} gave
\[
    b(\mathbf m)
    =
    [z_{01}^2z_{10}^2]\det C(z)
    =-2.
\]
Thus Lemma~\ref{lem:coeff-permanent} predicts
\[
    \per W_{\cD}=2!2!\cdot(-2)=-8,
\]
which is the direct permanent calculation.  The factor $\mathbf m!$ is exactly
the number of ways to attach the labels $d_0,d_1$ to the two occurrences of
$01$ and the labels $d_2,d_3$ to the two occurrences of $10$.
\end{example}

\begin{proof}[Proof of Lemma~\ref{lem:coeff-permanent}]
Fix the monomial
\[
    z^{\mathbf m}=\prod_{x\in H}z_x^{m_x}.
\]
Consider one contribution to its coefficient in
\[
    P_H(z)=\prod_{\rho\in\whH}
    \left(\sum_{x\in H}\eps_\rho(x)z_x\right).
\]
Such a contribution chooses from every character factor, indexed by
$\rho\in\whH$, one term
\[
    \eps_\rho(x)z_x.
\]
Altogether, the expansion selects the difference value $x$ exactly $m_x$
times for every $x\in H$.

Now compare this choice with a term in the permanent of $W_{\cD}$.  Whenever
the expansion selects the value $x$, we must assign that occurrence to one of
the labeled columns of $W_{\cD}$ whose prescribed difference is $x$.  At the
first occurrence of $x$, there are $m_x$ possible labeled columns.  At the
second occurrence there remain $m_x-1$ possibilities, and so on.  Thus the
$m_x$ occurrences of $x$ give
\[
    m_x(m_x-1)\cdots1=m_x!
\]
different permanent terms.

The assignments for distinct difference values $x\in H$ are independent.
Hence every contribution to
\[
    [z^{\mathbf m}]P_H(z)=b(\mathbf m)
\]
appears exactly
\[
    \prod_{x\in H}m_x!=\mathbf m!
\]
times in $\per W_{\cD}$.

All these permanent terms have the same value, since the columns of
$W_{\cD}$ corresponding to the same prescribed difference $x$ are identical.
Therefore
\[
    \per W_{\cD}=\mathbf m!\,b(\mathbf m),
\]
as required.
\end{proof}

\subsection{The \texorpdfstring{$2$}{2}-adic recursion}

Our goal is to determine the exact $2$-adic valuation of
$\per W_{\cD}$.  We shall do this recursively on the dimension of $H$.
The intended reduction is from a character matrix of order $n$ to character
matrices of order $n/2$.  To obtain such a reduction, we pair character rows,
regroup the permanent terms according to induced pairings of the labeled
columns, and exploit a local $2\times2$ cancellation.  The surviving column
pairs collapse to their sums in an index-two subspace.  The crucial point is
that both the permanent structure and the zero-sum condition survive this
compression, so the resulting matrices are smaller instances of exactly the
same problem.

Throughout this subsection, let
\[
    \cD=[d_0,d_1,\ldots,d_{n-1}],
    \qquad
    d_i\in H,
\]
be a prescribed-difference matrix on $H$.  Zero differences are allowed.  This
is necessary because the compression step replaces two prescribed differences
by their sum, which may be zero even when both original differences are
nonzero.

Fix a nonzero character index $\eta\in\whH$ and define
\[
    H_0=\{x\in H:\langle\eta,x\rangle=0\},
    \qquad
    H_1=\{x\in H:\langle\eta,x\rangle=1\}.
\]
Thus $H_0=\ker\eta$ is a subspace of index two and $H_1$ is its other coset.
The first step is local: after pairing the rows $\rho$ and $\rho+\eta$, we
ask what happens when two labeled columns are assigned to this row pair.
The next lemma shows the exact dichotomy needed for the recursion.  Columns
from opposite sides of the $\eta$-cut cancel, while a same-side pair survives
and is remembered only through the sum of its two differences.

\begin{lemma}[Two-row cancellation]
\label{lem:two-row-cancellation}
For $\rho\in\whH$ and prescribed-difference indices $i,j\in[n]$, define
\[
B_\rho(i,j)=
\per
\begin{pmatrix}
\eps_\rho(d_i)&\eps_\rho(d_j)\\
\eps_{\rho+\eta}(d_i)&\eps_{\rho+\eta}(d_j)
\end{pmatrix}.
\]
Then
\[
B_\rho(i,j)
=
\eps_\rho(d_i+d_j)
\bigl(\eps_\eta(d_i)+\eps_\eta(d_j)\bigr).
\]
Consequently,
\[
    B_\rho(i,j)=0
\]
when $d_i$ and $d_j$ lie in different sets $H_0,H_1$, whereas if
$d_i,d_j\in H_t$, then
\[
    B_\rho(i,j)
    =
    2(-1)^t\eps_\rho(d_i+d_j).
\]
\end{lemma}

\begin{example}[A vanishing and a surviving four-entry block]
\label{ex:running-block}
Continue with the profile of Example~\ref{ex:running-hall}.  Take
\[
    H=\F_2^2,
    \qquad
    \eta=01,
    \qquad
    \rho=00.
\]
Then
\[
    H_0=\{00,10\},
    \qquad
    H_1=\{01,11\}.
\]
Pair one occurrence of $01\in H_1$ with one occurrence of $10\in H_0$.
The four entries cut out by the row pair $\{00,01\}$ and these two columns are
\[
\begin{pmatrix}
1&1\\
-1&1
\end{pmatrix}.
\]
There are two possible internal assignments of the two columns to the two rows,
and their contributions cancel:
\[
    1\cdot1+1\cdot(-1)=0.
\]
Thus a column pair taken from opposite sides of the character cut disappears
before the recursion begins.

By contrast, pair the two labeled occurrences of $01$, both in $H_1$.  The
corresponding four-entry block is
\[
\begin{pmatrix}
1&1\\
-1&-1
\end{pmatrix},
\]
and the two internal assignments now reinforce each other:
\[
    1\cdot(-1)+1\cdot(-1)=-2.
\]
The same-side pair survives with the factor $2$ predicted by
Lemma~\ref{lem:two-row-cancellation}.  This local cancellation is the first
half of the shrink step: incompatible column pairs vanish, while compatible pairs are
allowed to collapse to one smaller column.
\end{example}

\begin{proof}
Using
\[
    \eps_{\rho+\eta}(x)
    =
    \eps_\rho(x)\eps_\eta(x),
\]
we obtain
\begin{align*}
B_\rho(i,j)
&=
\eps_\rho(d_i)\eps_{\rho+\eta}(d_j)
+
\eps_\rho(d_j)\eps_{\rho+\eta}(d_i)\\
&=
\eps_\rho(d_i+d_j)
\bigl(\eps_\eta(d_i)+\eps_\eta(d_j)\bigr).
\end{align*}
If $d_i$ and $d_j$ lie on opposite sides of the character cut, the last
parenthesis is
\[
    1+(-1)=0.
\]
If both lie in $H_t$, then it is
\[
    2(-1)^t.
\]
\end{proof}

We now pair all rows in this way.  The involution
\[
    \rho\longmapsto\rho+\eta
\]
has no fixed points on $\whH$.  It therefore partitions the rows of
$W_{\cD}$ into $n/2$ pairs
\[
    \{\rho,\rho+\eta\}.
\]
Choose a transversal $R\subseteq\whH$ containing one representative from each
row pair.  A \emph{column pairing} $\cP$ is a partition of the labeled indices
$[n]$ into $n/2$ unordered pairs.

Before using the cancellation lemma, we must justify that grouping rows and
columns into pairs does not destroy the permutation condition in the
permanent.  The next lemma is the bookkeeping statement that makes the
compression exact: every permanent term determines a unique pairing of the
labeled columns and a bijection from row pairs to column pairs, while the
$2\times2$ permanent records the two internal assignments that were forgotten.

\begin{lemma}[Exact regrouping of the permanent]
\label{lem:exact-permanent-regrouping}
We have
\[
\per W_{\cD}
=
\sum_{\cP}
\sum_{\substack{\phi:R\to\cP\\ \phi\text{ bijective}}}
\prod_{\rho\in R}B_\rho(i_\rho,j_\rho),
\]
where
\[
    \phi(\rho)=\{i_\rho,j_\rho\},
\]
and the outer sum is over all column pairings $\cP$ of $[n]$.
\end{lemma}

\begin{proof}
A term of $\per W_{\cD}$ assigns the $n$ labeled columns bijectively to the
$n$ rows.  For each row pair
\[
    \{\rho,\rho+\eta\},
\]
forget the order of the two columns assigned to its two rows.  This produces
one unordered pair of labeled column indices.  Doing this for every row pair
produces a unique column pairing $\cP$ and a unique bijection
\[
    \phi:R\to\cP
\]
recording which column pair is assigned to each row pair.

Conversely, fix $(\cP,\phi)$.  For each $\rho\in R$, there are exactly two
ways to assign the two columns of $\phi(\rho)$ to the rows
$\rho$ and $\rho+\eta$.  The sum of these two assignments is precisely
\[
    B_\rho(i_\rho,j_\rho).
\]
The choices for distinct row pairs are independent.  Hence the sum of all
permanent terms corresponding to the fixed pair $(\cP,\phi)$ is
\[
    \prod_{\rho\in R}B_\rho(i_\rho,j_\rho).
\]
Every permanent term determines a unique $(\cP,\phi)$ and a unique internal
ordering inside each row pair, so no term is omitted or counted twice.
\end{proof}

We now retain only the column pairings that survive the local cancellation.
Put
\[
    n_t=|\{i\in[n]:d_i\in H_t\}|
    \qquad(t=0,1).
\]
A column pairing $\cP$ is \emph{$\eta$-compatible} if every pair
\[
    p=\{i,j\}\in\cP
\]
lies entirely on one side of the character cut:
\[
    d_i,d_j\in H_0
    \qquad\text{or}\qquad
    d_i,d_j\in H_1.
\]
Let
\[
    \Pi_\eta(\cD)
\]
denote the set of all $\eta$-compatible column pairings.

For
\[
    \cP\in\Pi_\eta(\cD)
\]
and
\[
    p=\{i,j\}\in\cP,
\]
put
\[
    s(p)=d_i+d_j.
\]
Since $p$ is compatible,
\[
    s(p)\in H_0.
\]
Define the compressed prescribed-difference matrix
\[
    \cD_{\cP}
    =
    [\,s(p):p\in\cP\,].
\]
Thus $\cD_{\cP}$ has $n/2$ labeled columns in $H_0$.  Its character matrix
over the smaller group $H_0$ is denoted by
\[
    W^{H_0}_{\cD_{\cP}}.
\]

We can now assemble the two ingredients.  Exact regrouping organizes the
permanent by labeled column pairings, and the two-row cancellation eliminates
every pairing that crosses the character cut.  What remains must still be
identified as a permanent of the same type on the smaller space $H_0$.
The next lemma is the recursive identity used in the induction; the second
part is equally important, because it shows that a zero-sum instance remains
zero-sum after compression.

\begin{lemma}[Permanent recursion]
\label{lem:compatible-pairing-reduction}
Assume that $n_1$ is even.  Then
\[
    \per W_{\cD}
    =
    2^{n/2}(-1)^{n_1/2}
    \sum_{\cP\in\Pi_\eta(\cD)}
    \per W^{H_0}_{\cD_{\cP}}.
\]
Moreover, for every
\[
    \cP\in\Pi_\eta(\cD),
\]
we have
\[
    |\cD_{\cP}|=\frac n2
\]
and
\[
    \sum_{p\in\cP}s(p)
    =
    \sum_{i=0}^{n-1}d_i.
\]
Hence, if $\cD$ is zero-sum, then every compressed matrix
$\cD_{\cP}$ is zero-sum in $H_0$.
\end{lemma}

Before proving the identity in general, the following example displays the
entire shrink step.  It shows simultaneously what happens to the rows, what
happens to the labeled columns, and why the resulting $4\times4$ matrix is
again a prescribed-difference character matrix rather than an arbitrary block
compression.

\begin{example}[A literal $8\times8$ to $4\times4$ shrink]
\label{ex:complete-shrink}
Let
\[
    H=\F_2^3,
    \qquad
    \eta=001,
\]
so
\[
    H_0=\{000,010,100,110\},
    \qquad
    H_1=\{001,011,101,111\}.
\]
Consider the zero-sum labeled difference list
\[
    \cD=[000,000,000,010,001,101,011,101]
\]
and the compatible labeled column pairing
\[
\begin{aligned}
    p_1&=\{d_0,d_1\}=\{000,000\},
    &s(p_1)&=000,\\
    p_2&=\{d_2,d_3\}=\{000,010\},
    &s(p_2)&=010,\\
    p_3&=\{d_4,d_5\}=\{001,101\},
    &s(p_3)&=100,\\
    p_4&=\{d_6,d_7\}=\{011,101\},
    &s(p_4)&=110.
\end{aligned}
\]
Pair the rows as
\[
    R_1=\{000,001\},\quad
    R_2=\{010,011\},\quad
    R_3=\{100,101\},\quad
    R_4=\{110,111\}.
\]

Figure~\ref{fig:literal-shrink} writes the $8\times8$ matrix not as sixty-four
unrelated signs, but as a $4\times4$ array of $2\times2$ blocks.  The block
in position $(u,v)$ uses the two rows of $R_u$ and the two labeled columns of
$p_v$.

\begin{figure}[t]
\centering
\resizebox{0.96\textwidth}{!}{%
\begin{tikzpicture}[
    x=0.9cm,
    y=0.9cm,
    block/.style={draw,minimum width=1.45cm,minimum height=1.0cm,inner sep=1pt},
    smalllabel/.style={font=\scriptsize},
    rowlabel/.style={font=\small},
    collabel/.style={font=\small},
    titlelabel/.style={font=\small},
    >=latex
]


\node[titlelabel] at (3.25,5.75)
{$W_{\cD}$ arranged as a $4\times4$ array of $2\times2$ blocks};

\node[smalllabel] at (3.25,5.43) {column pairs};
\node[collabel] at (0.8,5.05) {$p_1$};
\node[collabel] at (2.45,5.05) {$p_2$};
\node[collabel] at (4.10,5.05) {$p_3$};
\node[collabel] at (5.75,5.05) {$p_4$};

\node[smalllabel,rotate=90] at (-1.48,2.48) {row pairs};
\node[rowlabel] at (-0.82,4.35) {$R_1$};
\node[rowlabel] at (-0.82,3.10) {$R_2$};
\node[rowlabel] at (-0.82,1.85) {$R_3$};
\node[rowlabel] at (-0.82,0.60) {$R_4$};

\node[block] at (0.8,4.35) {$\begin{smallmatrix}1&1\\ 1&1\end{smallmatrix}$};
\node[block] at (2.45,4.35) {$\begin{smallmatrix}1&1\\ 1&1\end{smallmatrix}$};
\node[block] at (4.10,4.35) {$\begin{smallmatrix}1&1\\ -1&-1\end{smallmatrix}$};
\node[block] at (5.75,4.35) {$\begin{smallmatrix}1&1\\ -1&-1\end{smallmatrix}$};

\node[block] at (0.8,3.10) {$\begin{smallmatrix}1&1\\ 1&1\end{smallmatrix}$};
\node[block] at (2.45,3.10) {$\begin{smallmatrix}1&-1\\ 1&-1\end{smallmatrix}$};
\node[block] at (4.10,3.10) {$\begin{smallmatrix}1&1\\ -1&-1\end{smallmatrix}$};
\node[block] at (5.75,3.10) {$\begin{smallmatrix}-1&1\\ 1&-1\end{smallmatrix}$};

\node[block] at (0.8,1.85) {$\begin{smallmatrix}1&1\\ 1&1\end{smallmatrix}$};
\node[block] at (2.45,1.85) {$\begin{smallmatrix}1&1\\ 1&1\end{smallmatrix}$};
\node[block] at (4.10,1.85) {$\begin{smallmatrix}1&-1\\ -1&1\end{smallmatrix}$};
\node[block] at (5.75,1.85) {$\begin{smallmatrix}1&-1\\ -1&1\end{smallmatrix}$};

\node[block] at (0.8,0.60) {$\begin{smallmatrix}1&1\\ 1&1\end{smallmatrix}$};
\node[block] at (2.45,0.60) {$\begin{smallmatrix}1&-1\\ 1&-1\end{smallmatrix}$};
\node[block] at (4.10,0.60) {$\begin{smallmatrix}1&-1\\ -1&1\end{smallmatrix}$};
\node[block] at (5.75,0.60) {$\begin{smallmatrix}-1&-1\\ 1&1\end{smallmatrix}$};

\draw[line width=0.8pt] (-0.10,4.88) -- (-0.25,4.88) -- (-0.25,0.07) -- (-0.10,0.07);
\draw[line width=0.8pt] (6.65,4.88) -- (6.80,4.88) -- (6.80,0.07) -- (6.65,0.07);

\node[smalllabel,align=center] at (3.25,-0.42)
{$8$ rows $\longrightarrow 4$ row pairs};
\node[smalllabel,align=center] at (3.25,-0.82)
{$8$ labeled columns $\longrightarrow 4$ column pairs};

\draw[->,line width=0.9pt] (7.65,2.40) -- (9.45,2.40);
\node[smalllabel,align=center] at (8.55,2.78)
{recursive shrink \\ step};
\node[smalllabel,align=center] at (8.55,2.01)
{take each $2\times2$ block\\ permanent};


\node[titlelabel] at (13.65,5.75) {$\cB_{\cP}$};

\node[smalllabel] at (13.65,5.43) {column pairs};
\node[collabel] at (12.00,5.05) {$p_1$};
\node[collabel] at (13.10,5.05) {$p_2$};
\node[collabel] at (14.20,5.05) {$p_3$};
\node[collabel] at (15.30,5.05) {$p_4$};

\node[smalllabel,rotate=90] at (10.75,2.48) {row pairs};
\node[rowlabel] at (11.22,4.35) {$R_1$};
\node[rowlabel] at (11.22,3.10) {$R_2$};
\node[rowlabel] at (11.22,1.85) {$R_3$};
\node[rowlabel] at (11.22,0.60) {$R_4$};

\draw[line width=0.8pt] (11.45,4.88) rectangle (15.85,0.07);
\foreach \x in {12.55,13.65,14.75}
    \draw[line width=0.8pt] (\x,4.88)--(\x,0.07);
\foreach \y in {3.675,2.4725,1.27}
    \draw[line width=0.8pt] (11.45,\y)--(15.85,\y);

\node at (12.00,4.28) {$2$};
\node at (13.10,4.28) {$2$};
\node at (14.20,4.28) {$-2$};
\node at (15.30,4.28) {$-2$};

\node at (12.00,3.08) {$2$};
\node at (13.10,3.08) {$-2$};
\node at (14.20,3.08) {$-2$};
\node at (15.30,3.08) {$2$};

\node at (12.00,1.87) {$2$};
\node at (13.10,1.87) {$2$};
\node at (14.20,1.87) {$2$};
\node at (15.30,1.87) {$2$};

\node at (12.00,0.67) {$2$};
\node at (13.10,0.67) {$-2$};
\node at (14.20,0.67) {$2$};
\node at (15.30,0.67) {$-2$};

\node[smalllabel,align=center] at (13.65,-0.42)
{one entry for each pair $(R_u,p_v)$};
\node[smalllabel,align=center] at (13.65,-0.82)
{same permanent problem on the smaller space $H_0$};

\end{tikzpicture}%
}
\caption{The recursive shrink step for one fixed compatible labeled column
pairing $\cP$.  The $8\times8$ character matrix is displayed as a $4\times4$
array of $2\times2$ blocks, with rows indexed by the row pairs $R_u$ and block
columns indexed by the column pairs $p_v$.  Replacing every $2\times2$ block by
its permanent produces the compressed $4\times4$ matrix $\cB_{\cP}$.  After
extracting the common column factors, this matrix is exactly the
prescribed-difference character matrix $W^{H_0}_{\cD_{\cP}}$ on the smaller
space $H_0=\ker\eta$.  Thus the shrink step is the inductive step: the original
permanent problem is replaced by a smaller permanent problem of the same form.}
\label{fig:literal-shrink}
\end{figure}
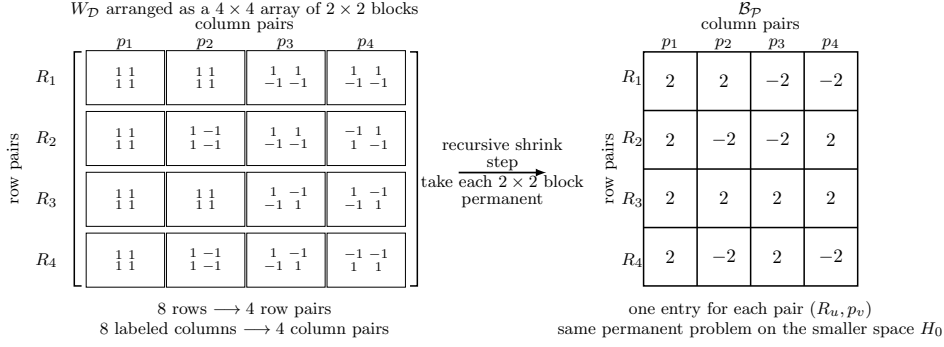

The figure also makes the row and column conditions transparent.  A permanent
term of the original matrix uses every one of the eight rows and every one of
the eight labeled columns exactly once.  Once the row pairing and the fixed
column pairing $\cP$ are imposed, this global bijection has two levels.

First, at the coarse level, every row pair $R_u$ must be matched with exactly
one column pair $p_v$, and every column pair must be used exactly once.  Thus
the coarse choices form a permutation of four objects; equivalently, they are
exactly the terms of $\per\cB_{\cP}$.  Second, inside every chosen
$2\times2$ block there are exactly two ways to assign its two labeled columns
to its two rows.  Their sum is the permanent of that block.  Therefore the
contribution of the fixed pairing $\cP$ to $\per W_{\cD}$ is exactly
\[
    \per\cB_{\cP}.
\]
The row and column conditions have not disappeared in the shrink step: they have
descended from individual rows and columns to row pairs and column pairs.

For example,
\[
\per
\begin{pmatrix}
1&1\\
-1&-1
\end{pmatrix}
=-2.
\]
This is a surviving block.  Its column pair lies entirely in $H_1$, and after
the common factor $-2$ is removed the block records the character value at the
compressed sum of the two columns.

By contrast, suppose one paired a column $d_i\in H_0$ with a column
$d_j\in H_1$.  In the first row pair one obtains, for example,
\[
\per
\begin{pmatrix}
1&1\\
1&-1
\end{pmatrix}
=0.
\]
The same cancellation occurs in every row pair.  Hence the entire would-be
block column is zero: an incompatible column pair is annihilated before it can
become a column of the smaller permanent.

We now identify the smaller matrix.  For a compatible pair
$p=\{i,j\}$, its two differences lie on the same side of the $\eta$-cut, so
\[
    s(p)=d_i+d_j\in H_0.
\]
Thus each block column $p_v$ becomes one labeled column indexed by the
compressed difference $s(p_v)$.  In the present example,
\[
    \cD_{\cP}=[000,010,100,110].
\]
Even if two different column pairs had the same sum, they would remain two
distinct labeled compressed columns, so the column condition of the permanent
would still be preserved.

On the row side, $\rho$ and $\rho+\eta$ agree on $H_0=\ker\eta$.  Hence each
row pair $R_u$ has one common restricted character on $H_0$.  The four row
pairs give the four distinct characters of $H_0$, exactly once.  Therefore,
after extracting the block-column factors
\[
    2,\quad2,\quad-2,\quad-2,
\]
we obtain
\[
\cB_{\cP}
=
\underbrace{
\begin{pmatrix}
 1& 1& 1& 1\\
 1&-1& 1&-1\\
 1& 1&-1&-1\\
 1&-1&-1& 1
\end{pmatrix}}_{\displaystyle W^{H_0}_{\cD_{\cP}}}
\begin{pmatrix}
2&0&0&0\\
0&2&0&0\\
0&0&-2&0\\
0&0&0&-2
\end{pmatrix}.
\]
This displays both halves of the compression:
\[
\boxed{
\begin{array}{c}
\{\rho,\rho+\eta\}
\longmapsto
\rho|_{H_0},
\\[1mm]
\{d_i,d_j\}
\longmapsto
d_i+d_j\in H_0.
\end{array}}
\]
The first map sends the four row pairs to all four character rows of the
smaller space; the second sends the four labeled column pairs to the four
labeled compressed columns.  The compressed list also preserves the zero-sum
condition:
\[
    \sum_{v=1}^{4}s(p_v)
    =
    \sum_{i=0}^{7}d_i
    =
    0.
\]
Thus the smaller matrix belongs to the same zero-sum
prescribed-difference character family as the original matrix.  Numerically,
\[
    \per W^{H_0}_{\cD_{\cP}}=8,
\]
and this fixed pairing contributes
\[
    \per\cB_{\cP}
    =
    2^4(-1)^2\cdot8
    =
    128.
\]
The full recursion sums the corresponding smaller permanents over all
compatible labeled column pairings $\cP$.  The example therefore shows the
inductive mechanism in concrete form: order $8$ has been reduced to the same
permanent problem at order $4$.
\end{example}

The example shows the form of the recursive reduction.  The proof below
formalizes the same two-level bookkeeping for arbitrary $n$.  The only global
sign comes from the compatible column pairs lying in $H_1$; there are exactly
$n_1/2$ such pairs.

\begin{proof}
By Lemma~\ref{lem:exact-permanent-regrouping}, the permanent is a sum over all
column pairings.  If $\cP$ is not $\eta$-compatible, then it contains a pair
\[
    \{i,j\}
\]
with $d_i$ and $d_j$ on opposite sides of the character cut.  By
Lemma~\ref{lem:two-row-cancellation},
\[
    B_\rho(i,j)=0
\]
for every $\rho\in R$.  Hence the contribution of $\cP$ is zero.

Fix
\[
    \cP\in\Pi_\eta(\cD).
\]
For
\[
    p=\{i,j\}\in\cP,
\]
let
\[
    t(p)=
    \begin{cases}
        0,&d_i,d_j\in H_0,\\
        1,&d_i,d_j\in H_1.
    \end{cases}
\]
Lemma~\ref{lem:two-row-cancellation} gives
\[
    B_\rho(i,j)
    =
    2(-1)^{t(p)}\eps_\rho(s(p)).
\]
Therefore the contribution of $\cP$ in
Lemma~\ref{lem:exact-permanent-regrouping} is
\[
2^{n/2}
\left(
    \prod_{p\in\cP}(-1)^{t(p)}
\right)
\sum_{\substack{\phi:R\to\cP\\ \phi\text{ bijective}}}
\prod_{\rho\in R}
\eps_\rho\bigl(s(\phi(\rho))\bigr).
\]
Exactly $n_1/2$ pairs of $\cP$ lie in $H_1$.  Hence
\[
    \prod_{p\in\cP}(-1)^{t(p)}
    =
    (-1)^{n_1/2}.
\]

For $x\in H_0$,
\[
    \eps_{\rho+\eta}(x)=\eps_\rho(x).
\]
Thus restriction of characters identifies
\[
    \whH/\langle\eta\rangle
    \cong
    \widehat{H_0}.
\]
Since $R$ contains one representative from each pair
$\{\rho,\rho+\eta\}$, the restrictions of the characters indexed by $R$ are
exactly all characters of $H_0$, once each.  Consequently,
\[
\sum_{\substack{\phi:R\to\cP\\ \phi\text{ bijective}}}
\prod_{\rho\in R}
\eps_\rho\bigl(s(\phi(\rho))\bigr)
=
\per W^{H_0}_{\cD_{\cP}}.
\]
The contribution of $\cP$ is therefore
\[
    2^{n/2}(-1)^{n_1/2}
    \per W^{H_0}_{\cD_{\cP}}.
\]
Summing over all compatible pairings gives
\[
    \per W_{\cD}
    =
    2^{n/2}(-1)^{n_1/2}
    \sum_{\cP\in\Pi_\eta(\cD)}
    \per W^{H_0}_{\cD_{\cP}}.
\]

Finally, $\cP$ contains exactly $n/2$ pairs, so
\[
    |\cD_{\cP}|=\frac n2.
\]
Also,
\[
    \sum_{p\in\cP}s(p)
    =
    \sum_{\{i,j\}\in\cP}(d_i+d_j)
    =
    \sum_{i=0}^{n-1}d_i,
\]
because every labeled index occurs in exactly one pair.
\end{proof}

Lemma~\ref{lem:compatible-pairing-reduction} is the promised dimension-halving
step.  A compatible row pair becomes one character of $H_0$, a compatible
labeled column pair becomes one labeled difference in $H_0$, and the total
sum of the compressed differences is unchanged.  We may therefore apply the
same statement again on the smaller space.

The recursion is now strong enough to determine the exact valuation.  The
remaining issue is cancellation between the different compatible column
pairings.  For a zero-sum list, both sides of every character cut contain an
even number of labeled columns, so the number of compatible pairings is odd.
Inductively, every smaller permanent has the same exact power of $2$ and an
odd normalized part.  An odd number of odd normalized contributions cannot
cancel modulo $2$.  This is the parity mechanism behind the next lemma.

\begin{lemma}[Full character-matrix permanent valuation]
\label{lem:full-character-permanent}
Let
\[
    H\cong\F_2^r,
    \qquad
    n=|H|,
\]
and let
\[
    \cD=[d_0,d_1,\ldots,d_{n-1}]
\]
be a prescribed-difference matrix over $H$, where zero differences are
allowed.  If
\[
    \sum_{i=0}^{n-1}d_i=0,
\]
then
\[
    \vTwo(\per W_{\cD})=n-1.
\]
In particular,
\[
    \per W_{\cD}\neq0.
\]
\end{lemma}

\begin{proof}
We argue by induction on $r$.

For $r=0$, the space $H$ has one element and
\[
    n=1.
\]
The only zero-sum prescribed-difference matrix is
\[
    \cD=[0],
\]
and
\[
    W_{\cD}=[1].
\]
Hence
\[
    \vTwo(\per W_{\cD})=0=n-1.
\]

Assume $r\geq1$ and that the statement holds in dimension $r-1$.  Fix a
nonzero
\[
    \eta\in\whH.
\]
Since
\[
    \sum_{i=0}^{n-1}d_i=0,
\]
we have
\[
    n_1
    \equiv
    \left\langle
        \eta,
        \sum_{i=0}^{n-1}d_i
    \right\rangle
    \equiv0\pmod2.
\]
Thus Lemma~\ref{lem:compatible-pairing-reduction} gives
\[
    \per W_{\cD}
    =
    2^{n/2}(-1)^{n_1/2}
    \sum_{\cP\in\Pi_\eta(\cD)}
    \per W^{H_0}_{\cD_{\cP}}.
\]

For every
\[
    \cP\in\Pi_\eta(\cD),
\]
the compressed matrix $\cD_{\cP}$ has
\[
    \frac n2=|H_0|
\]
columns and is zero-sum in $H_0$.  By the induction hypothesis,
\[
    \vTwo\!\left(
        \per W^{H_0}_{\cD_{\cP}}
    \right)
    =
    \frac n2-1.
\]
Write
\[
    \per W^{H_0}_{\cD_{\cP}}
    =
    2^{n/2-1}u_{\cP},
\]
where $u_{\cP}$ is odd.

Since $n_1$ is even and $n$ is even,
\[
    n_0=n-n_1
\]
is also even.  The number of compatible pairings is
\[
    |\Pi_\eta(\cD)|
    =
    (n_0-1)!!(n_1-1)!!,
\]
with
\[
    (-1)!!=1.
\]
Both factors are odd.  Hence
\[
    |\Pi_\eta(\cD)|
\]
is odd.

Therefore
\[
    \sum_{\cP\in\Pi_\eta(\cD)}u_{\cP}
\]
is a sum of an odd number of odd integers and is itself odd.  Substituting in
the recursion gives
\[
    \per W_{\cD}
    =
    2^{n-1}(-1)^{n_1/2}
    \sum_{\cP\in\Pi_\eta(\cD)}u_{\cP},
\]
where the final sum is odd.  Thus
\[
    \vTwo(\per W_{\cD})=n-1.
\]
\end{proof}

The zero-sum hypothesis is exactly what closes the recursion: it makes the
character cut even on both sides and is preserved by every compatible
compression
\[
    \cD\longmapsto\cD_{\cP}.
\]

We now return from the repeated-column permanent to the determinant
coefficient that encodes Hall realizations.  The correspondence
$\mathbf m!\,b(\mathbf m)=\per W_{\cD}$ converts the exact permanent valuation
into the desired multinomial valuation.  The nonzero-sum case is handled
separately by a character-shift involution, which gives the algebraic
counterpart of the obvious zero-sum obstruction for matchings.

For a nonnegative integer $a$, let $s_2(a)$ denote the number of ones in the
binary expansion of $a$.

\begin{lemma}[Full character coefficient lemma]
\label{lem:full-character-coeff}
Let
\[
    |\mathbf m|=n
\]
and let
\[
    b(\mathbf m)=[z^{\mathbf m}]P_H(z).
\]
Then
\begin{enumerate}[label=\textup{(\roman*)}]
    \item if $\sigma(\mathbf m)\neq0$, then
    \[
        b(\mathbf m)=0;
    \]

    \item if $\sigma(\mathbf m)=0$, then
    \[
        \vTwo(b(\mathbf m))
        =
        \vTwo\!\left(\frac{n!}{\mathbf m!}\right).
    \]
\end{enumerate}
\end{lemma}

\begin{proof}
Suppose first that
\[
    S=\sigma(\mathbf m)\neq0.
\]
Choose
\[
    \eta\in\whH
\]
such that
\[
    \langle\eta,S\rangle=1.
\]

A contribution to
\[
    [z^{\mathbf m}]P_H(z)
\]
is specified by choosing, for every $\rho\in\whH$, a value
\[
    x_\rho\in H
\]
such that each $x\in H$ occurs exactly $m_x$ times.  Its sign is
\[
    \prod_{\rho\in\whH}\eps_\rho(x_\rho).
\]
Apply the involution
\[
    x_\rho\longmapsto x_{\rho+\eta}.
\]
The multiplicities are unchanged, while the new sign is
\begin{align*}
\prod_{\rho\in\whH}\eps_\rho(x_{\rho+\eta})
&=
\prod_{\rho\in\whH}\eps_{\rho+\eta}(x_\rho)\\
&=
\left(
\prod_{\rho\in\whH}\eps_\rho(x_\rho)
\right)
\left(
\prod_{\rho\in\whH}\eps_\eta(x_\rho)
\right).
\end{align*}
Since the value $x$ occurs $m_x$ times,
\[
    \prod_{\rho\in\whH}\eps_\eta(x_\rho)
    =
    \prod_{x\in H}\eps_\eta(x)^{m_x}
    =
    (-1)^{\langle\eta,S\rangle}
    =
    -1.
\]
Thus the involution pairs every contribution with one of opposite sign.  Hence
\[
    b(\mathbf m)=0.
\]

Now suppose
\[
    \sigma(\mathbf m)=0.
\]
Choose any prescribed-difference matrix
\[
    \cD=[d_0,\ldots,d_{n-1}]
\]
with multiplicity vector $\mathbf m$.  Then
\[
    \sum_{i=0}^{n-1}d_i
    =
    \sigma(\mathbf m)
    =
    0.
\]
By Lemma~\ref{lem:coeff-permanent} and
Lemma~\ref{lem:full-character-permanent},
\[
    \vTwo(\mathbf m!)
    +
    \vTwo(b(\mathbf m))
    =
    n-1.
\]
Since
\[
    n=2^r,
\]
Legendre's formula gives
\[
    \vTwo(n!)
    =
    n-s_2(n)
    =
    n-1.
\]
Therefore
\[
    \vTwo(b(\mathbf m))
    =
    n-1-\vTwo(\mathbf m!)
    =
    \vTwo\!\left(\frac{n!}{\mathbf m!}\right).
\]
\end{proof}

\begin{proof}[Proof of Theorem~\ref{thm:full-coefficient}]
By Lemma~\ref{lem:character-diagonalization},
\[
    \det C(z)=P_H(z).
\]
The result is therefore exactly
Lemma~\ref{lem:full-character-coeff}.
\end{proof}

\begin{proof}[Proof of Theorem~\ref{thm:binary-hall}]
Let $\mathbf m$ be the multiplicity vector of $\cD$.  Since $\cD$ is
zero-sum,
\[
    \sigma(\mathbf m)=0.
\]
Theorem~\ref{thm:full-coefficient} gives
\[
    [z^{\mathbf m}]\det C(z)\neq0.
\]
By Lemma~\ref{lem:circulant-certificate}, there exists a permutation
\[
    \pi:H\to H
\]
whose realized difference multiplicities are exactly $\mathbf m$.  Hence
$\cD$ has a realization.
\end{proof}

\subsection{Back to crossing matchings}

We now state the consequence of the Binary Hall theorem in the original
prescribed-difference formulation.  The same hyperplane case is already known
from the Hall-based batch-code argument of \cite{HollmannEtAl2023}.

\begin{corollary}[Hyperplane prescribed-difference matching]
\label{cor:hyperplane-matching}
Let $H\leq\F_2^s$ be a hyperplane and let
\[
    \alpha\in\F_2^s\setminus H.
\]
Suppose
\[
    M=[v_0,\ldots,v_{n-1}],
    \qquad
    v_i\in H+\alpha,
    \qquad
    n=|H|,
\]
and
\[
    \sum_{i=0}^{n-1}v_i=0.
\]
Then $M$ has a Hadamard pair solution.
\end{corollary}

\begin{proof}
Let
\[
    \cD=[d_0,\ldots,d_{n-1}],
    \qquad
    d_i=v_i+\alpha\in H,
\]
be the prescribed-difference matrix associated with $M$.

Since $n$ is even,
\[
    \sum_{i=0}^{n-1}d_i
    =
    \sum_{i=0}^{n-1}(v_i+\alpha)
    =
    \sum_{i=0}^{n-1}v_i
    =
    0.
\]
By Theorem~\ref{thm:binary-hall}, $\cD$ has a realization.  By the equivalence
established at the beginning of this section, the corresponding crossing pairs
between $H$ and $H+\alpha$ form a Hadamard pair solution for $M$.
\end{proof}

The Hall case is therefore encoded by the full group circulant.  The
advantage of the determinant formulation is that deleting vertices does not
require a new matching polynomial: it simply replaces the full circulant by a
minor of the same matrix.  This is the mechanism used in the two-hole theorem.

\section{The Two-Hole Theorem}
\label{sec:two-hole}

Fix $H\leq\F_2^s$ and $\alpha\notin H$.  We now pass from the full Hall
permutation to the first punctured case.  The similarity with Hall is
substantial: $n-2$ requests are still crossing, so after translating
$H+\alpha$ by $\alpha$ the residual problem is still a prescribed-difference
bijection between two copies of $H$.  The essential difference is that these
copies are no longer complete.  The two internal requests must consume two
vertices on each affine side, and their locations are not prescribed.

The terminal obstruction described in Section~\ref{sec:prelim} is therefore
not a failure to realize crossing requests locally.  The missing
ingredient is simultaneous control of the two prescribed internal differences
at the final boundary.  The character-minor method addresses exactly this
point.  The Hall circulant remains the correct ambient matrix, but the relevant
object is now an unknown complementary minor rather than the full determinant.
We shall prove that some locations of the four deleted vertices, with the two
prescribed internal differences, always produce a nonzero coefficient.

\begin{theorem}[Two-hole theorem]
\label{thm:two-hole}
Let $s\geq2$, and let $M$ be a zero-sum two-hole instance with respect to a
hyperplane $H\leq\F_2^s$.  Then $M$ has a Hadamard pair solution.
\end{theorem}

\begin{corollary}[Exactly two even-weight requests]
\label{cor:two-even}
Let $M=[v_0,\ldots,v_{2^{s-1}-1}]$ have nonzero columns and zero column sum.
If exactly two columns of $M$, counted with multiplicity, have even Hamming
weight, then $M$ has a Hadamard pair solution.
\end{corollary}

\begin{proof}
Apply Theorem~\ref{thm:two-hole} to the parity hyperplane
$H_{\rm even}=\{x\in\F_2^s:\wt(x)\equiv0\pmod2\}$.
\end{proof}

The proof of Theorem~\ref{thm:two-hole} has three steps.  First, candidate
locations of the two internal edges are encoded by deleting two rows and two
columns from the Hall circulant.  Second, complementary-minor expansion moves
these deleted vertices to the character side, where $2\times2$ Walsh minors
filter the character differences compatible with the two prescribed internal
directions.  Finally, the dependence on the unknown boundary locations becomes
a Walsh transform; its invertibility forces at least one boundary choice with a
nonzero deleted-circulant coefficient.

Let the two internal requests be $p,q\in H\setminus\{0\}$, counted with
multiplicity.  By the hyperplane count in Section~\ref{sec:prelim}, one
internal edge lies on each affine side.  After translating the right side by
$\alpha$, the $n-2$ crossing requests must therefore be realized by a
bijection between two copies of $H$ with two vertices deleted from each.  The
next subsection develops the minor of $C(z)$ that records exactly such
bijections.

\subsection{Deleted circulants and complementary Walsh minors}
\label{subsec:deleted-framework}

For $A,B\subseteq H$ with $|A|=|B|=h$, write
$A^c=H\setminus A$ and $B^c=H\setminus B$.  The minor
$C(z)_{A^c,B^c}$ records bijections from $A^c$ to $B^c$ in exactly the same
way that the full circulant records permutations of $H$.

\begin{lemma}[Deleted-circulant certificate]
\label{lem:deleted-certificate}
Let $|A|=|B|=h$, and let $\mathbf m$ be a multiplicity profile with
\[
    |\mathbf m|=n-h.
\]
If
\[
    [z^{\mathbf m}]\det C(z)_{A^c,B^c}\neq0,
\]
then there exists a bijection
\[
    \pi:A^c\to B^c
\]
whose realized difference multiplicities are $\mathbf m$.
\end{lemma}

\begin{proof}
After fixing orders on $A^c$ and $B^c$, the determinant expansion is
\[
    \det C(z)_{A^c,B^c}
    =
    \sum_{\pi:A^c\overset{\sim}{\to}B^c}
    \sgn_{A,B}(\pi)
    \prod_{a\in A^c}z_{a+\pi(a)}.
\]
Thus a nonzero coefficient of $z^{\mathbf m}$ implies that at least one
bijection with difference profile $\mathbf m$ occurs.
\end{proof}

\begin{example}[A two-hole instance as an incomplete Hall permutation]
\label{ex:running-two-hole}
Let
\[
    H=\F_2^2=\{00,01,10,11\},
    \qquad
    \alpha=100,
\]
where we identify $H$ with the last two coordinates of $\F_2^3$.  Take the two
internal requests
\[
    p=01,
    \qquad
    q=10,
\]
and the two crossing requests
\[
    \alpha+00=100,
    \qquad
    \alpha+11=111.
\]
Their total sum is zero.  Choose
\[
    A=\{00,01\},
    \qquad
    B=\{00,10\}.
\]
The set $A$ is an internal $p$-edge on the left, and after translating the
right affine side back to $H$, the set $B$ represents an internal $q$-edge on
the right.  The remaining Hall problem is the incomplete bijection
\[
    A^c=\{10,11\}
    \longrightarrow
    B^c=\{01,11\}.
\]
The corresponding deleted circulant is
\[
    C(z)_{A^c,B^c}
    =
    \begin{pmatrix}
        z_{11}&z_{01}\\
        z_{10}&z_{00}
    \end{pmatrix},
\]
so
\[
    [z_{00}z_{11}]\det C(z)_{A^c,B^c}=1.
\]
The nonzero coefficient is witnessed by the bijection
\[
    10\mapsto01,
    \qquad
    11\mapsto11,
\]
whose reduced differences are $11$ and $00$.  Restoring the affine shift and
the two internal edges gives the perfect matching
\[
    \{000,001\},\quad
    \{100,110\},\quad
    \{010,101\},\quad
    \{011,111\},
\]
with prescribed differences $001,010,111,100$.  Thus the deleted minor is
literally the Hall coefficient matrix after the two holes have removed two
vertices from each side.
\end{example}

For
\[
    Q\subseteq\whH,
\]
define
\[
    \Theta_Q(z)
    =
    \prod_{\rho\in\whH\setminus Q}L_\rho(z).
\]
Whenever
\[
    |\mathbf m|=n-|Q|,
\]
put
\[
    b_Q(\mathbf m)
    =
    [z^{\mathbf m}]\Theta_Q(z).
\]

For Hall, diagonalization replaced the full circulant determinant by a product
of all character forms.  Here the rows and columns deleted by $A$ and $B$ must
remain visible.  Cauchy--Binet and Jacobi's complementary-minor identity do
exactly this: they transfer the deleted boundary to small Walsh minors
$W_{Q,A}$ and $W_{Q,B}$, while the remaining character factors form
$\Theta_Q(z)$.  The following identity is the bridge from a punctured Hall
matrix to deleted character coefficients.

\begin{lemma}[Complementary character-minor expansion]
\label{lem:general-complementary-minor}
Fix global orders on $H$ and $\whH$.  Let $A,B\subseteq H$ satisfy
\[
    |A|=|B|=h.
\]
There exists a sign
\[
    \sigma(A,B)\in\{\pm1\}
\]
such that
\begin{equation}
\label{eq:general-complementary-minor}
\det C(z)_{A^c,B^c}
=
\frac{\sigma(A,B)}{n^h}
\sum_{Q\in\binom{\whH}{h}}
\det W_{Q,A}\,
\det W_{Q,B}\,
\Theta_Q(z).
\end{equation}
The sign $\sigma(A,B)$ is independent of $Q$.
\end{lemma}

\begin{proof}
Put
\[
    I=A^c,
    \qquad
    J=B^c.
\]
Apply Cauchy--Binet to
\[
    C(z)=\frac1nW^{\mathsf T}\Lambda(z)W.
\]
For every $R\subseteq\whH$ of size $n-h$, the corresponding term is
\[
    n^{-(n-h)}
    \det W_{R,I}\,
    \det W_{R,J}\,
    \prod_{\rho\in R}L_\rho(z).
\]
Let
\[
    Q=\whH\setminus R,
    \qquad
    |Q|=h.
\]
Jacobi's complementary-minor identity and
\[
    W^{-1}=n^{-1}W^{\mathsf T}
\]
give
\[
    \det W_{R,I}
    =
    \sigma_1(R,I)\det(W)n^{-h}\det W_{Q,A},
\]
and
\[
    \det W_{R,J}
    =
    \sigma_2(R,J)\det(W)n^{-h}\det W_{Q,B}.
\]
The $R$-dependent complementary-minor sign occurs twice, so the product of the
two signs depends only on $A$ and $B$; denote it by $\sigma(A,B)$.  Since
\[
    \det(W)^2=n^n,
\]
the scalar factor is
\[
    n^{-(n-h)}n^n n^{-2h}=n^{-h}.
\]
Finally,
\[
    \prod_{\rho\in R}L_\rho(z)=\Theta_Q(z),
\]
which proves the formula.
\end{proof}

For $h=0$, Lemma~\ref{lem:general-complementary-minor} is the Hall
factorization
\[
    \det C(z)=P_H(z).
\]
For the two-hole theorem we now specialize to $h=2$.  Thus no new matching
polynomial is introduced: the first punctured Hall problem is encoded by
complementary minors of the same circulant used in the Hall case.

\subsection{The two-deleted coefficient}

After deleting two character factors, the row-pair compression no longer
produces a full Walsh matrix: one obtains a sign matrix of odd order
$2^q-1$.  To control the deleted coefficient uniformly, we therefore need a
valuation statement that holds for \emph{every} sign matrix of this order,
not only for a Sylvester--Hadamard matrix.  The following elementary lemma is
the reason the two-deleted calculation remains exact.

\begin{lemma}[Sign permanents of order $2^q-1$]\label{lem:sign-permanent}
Let $q\geq1$, let $m=2^q-1$, and let $A\in\{\pm1\}^{m\times m}$.
Then
\[
\vTwo(\per A)=\vTwo(m!)=m-q.
\]
In particular, $\per A\neq0$.
\end{lemma}

\begin{proof}
Write
\[
A=J-2B,
\]
where $J$ is the all-one matrix and $B$ is a $0$--$1$ matrix.  Expanding
the permanent according to the positions in which the $-2B$ term is used
gives
\begin{equation}\label{eq:per-J-2B}
\per(J-2B)
=
\sum_{k=0}^{m}(-2)^k(m-k)!
\sum_{\substack{I,J'\subseteq[m]\\|I|=|J'|=k}}
\per B[I,J'].
\end{equation}
The $k=0$ term is $m\mathord{!}$.  For every $k\geq1$, Legendre's formula gives
\[
\vTwo\bigl(2^k(m-k)!\bigr)
=
k+\vTwo((m-k)!)
=
m-s_2(m-k),
\]
where $s_2(a)$ denotes the number of ones in the binary expansion of $a$.
Since
\[
0\leq m-k\leq2^q-2,
\]
we have $s_2(m-k)\leq q-1$.  Therefore every term in
\eqref{eq:per-J-2B} with $k\geq1$ is divisible by
\[
2^{m-q+1}.
\]
On the other hand,
\[
\vTwo(m!)=m-s_2(m)=m-q.
\]
Hence
\[
\per A\equiv m!\pmod{2^{m-q+1}},
\]
which proves the claimed exact valuation.
\end{proof}

\begin{remark}
Lemma~\ref{lem:sign-permanent} is a statement about \emph{every} sign
matrix of the special odd order $2^q-1$.  Its proof above is self-contained.
For a related $2$-adic study of permanents of Sylvester--Hadamard matrices,
see Chabaud~\cite{Chabaud2018}.
\end{remark}

For distinct characters $\xi,\psi\in\whH$, abbreviate
\[
    \Theta_{\xi,\psi}(z)
    =
    \Theta_{\{\xi,\psi\}}(z),
    \qquad
    b_{\xi,\psi}(\mathbf m)
    =
    b_{\{\xi,\psi\}}(\mathbf m).
\]
Our next goal is to understand which deleted character pairs can contribute
to the prescribed crossing profile.  For a deleted pair
$\{\xi,\psi\}$, only its difference $\eta=\xi+\psi$ will matter.  We need two
facts at once: an exact criterion for vanishing and an exact valuation for every
surviving coefficient.  This is the main arithmetic input of the two-hole
proof.

\begin{lemma}[Two-deleted Walsh coefficient lemma]\label{lem:two-deleted-walsh}
Let $\xi\neq\psi$, let
\[
\eta=\xi+\psi,
\qquad
S=\sigma(\mathbf m),
\qquad
|\mathbf m|=n-2.
\]
Then
\[
b_{\xi,\psi}(\mathbf m)=0
\quad\Longleftrightarrow\quad
\langle\eta,S\rangle=1.
\]
Moreover, if $\langle\eta,S\rangle=0$, then
\begin{equation}\label{eq:walsh-coeff-valuation}
\vTwo\bigl(b_{\xi,\psi}(\mathbf m)\bigr)
=
\vTwo\left(\frac{(n-2)!}{\mathbf m!}\right).
\end{equation}
In particular, the coefficient is nonzero.
\end{lemma}

\begin{proof}
If $n=2$, then $\whH=\{\xi,\psi\}$, the only profile of size $n-2$ is the empty profile, and $\Theta_{\xi,\psi}(z)=1$.  Thus $S=0$ and the assertion is immediate.  Hence assume $n\geq4$.

Create an $(n-2)\times(n-2)$ matrix
\[
W_{\mathbf m}^{\xi,\psi}
\]
whose rows are indexed by
$\whH\setminus\{\xi,\psi\}$ and whose columns consist of $m_x$
labeled copies of the Walsh column indexed by $x$, for every $x\in H$.
The usual coefficient--permanent correspondence gives
\begin{equation}\label{eq:coeff-permanent}
\mathbf m!\,b_{\xi,\psi}(\mathbf m)
=
\per W_{\mathbf m}^{\xi,\psi}.
\end{equation}

We first prove the vanishing direction.  Translation of the row index by
$\eta$ permutes the set
\[
\whH\setminus\{\xi,\psi\},
\]
because the omitted pair $\{\xi,\psi\}$ is interchanged by
$\rho\mapsto\rho+\eta$.  Under this row permutation, the entry in every
column indexed by $x$ is multiplied by
\[
\eps_\eta(x)=(-1)^{\langle\eta,x\rangle}.
\]
Consequently,
\[
\per W_{\mathbf m}^{\xi,\psi}
=
(-1)^{\langle\eta,S\rangle}
\per W_{\mathbf m}^{\xi,\psi}.
\]
If $\langle\eta,S\rangle=1$, the permanent, and hence the coefficient in
\eqref{eq:coeff-permanent}, is zero.

Assume now that $\langle\eta,S\rangle=0$.  Translating all character
indices by $\xi$ changes the omitted pair from $\{\xi,\psi\}$ to
$\{0,\eta\}$ and only multiplies columns by signs.  Thus it changes the
permanent by an overall sign and does not change its $2$-adic valuation.
We may therefore assume that the omitted characters are $0$ and $\eta$.

Choose coordinates
\[
H=\F_2\times H_0,
\qquad
|H_0|=m_0=n/2,
\]
so that
\[
\eps_\eta(t,x)=(-1)^t.
\]
The remaining characters occur in the $m_0-1$ row pairs
\[
(0,a),(1,a),
\qquad
0\neq a\in\whH_0.
\]
Let $n_t$ be the number of labeled columns of the form $(t,x)$.  The
assumption $\langle\eta,S\rangle=0$ says that $n_1$ is even.  Since
$n_0+n_1=n-2$ is even, $n_0$ is even as well.

Consider one row pair $(0,a),(1,a)$ and two columns
$(t_1,x_1),(t_2,x_2)$.  The corresponding $2\times2$ permanent equals
\begin{align*}
&\eps_{(0,a)}(t_1,x_1)\eps_{(1,a)}(t_2,x_2)
+
\eps_{(0,a)}(t_2,x_2)\eps_{(1,a)}(t_1,x_1)
\\
&\hspace{2cm}=
\begin{cases}
0, & t_1\neq t_2,\\[2mm]
2(-1)^t\eps_a(x_1+x_2), & t_1=t_2=t.
\end{cases}
\end{align*}
Therefore, when the permanent of $W_{\mathbf m}^{0,\eta}$ is expanded by
first grouping the two columns assigned to each row pair, only pairings of
the $t=0$ columns among themselves and of the $t=1$ columns among
themselves survive.

Fix such a pairing $\mathcal P$ of the labeled columns.  After the factor
$2$ is extracted from each of the $m_0-1$ row pairs, the contribution of
$\mathcal P$ is
\begin{equation}\label{eq:pairing-contribution}
(-1)^{n_1/2}2^{m_0-1}\per A_{\mathcal P},
\end{equation}
where $A_{\mathcal P}$ is a sign matrix of order
\[
m_0-1=2^{r-1}-1.
\]
Indeed, its rows are indexed by $a\in\whH_0\setminus\{0\}$, its
columns are indexed by the pairs $\{(t,x),(t,x')\}\in\mathcal P$, and its
corresponding entry is
\[
\eps_a(x+x')\in\{\pm1\}.
\]

By Lemma~\ref{lem:sign-permanent},
\[
\vTwo(\per A_{\mathcal P})
=
(m_0-1)-(r-1)
=
m_0-r.
\]
Hence every contribution in \eqref{eq:pairing-contribution} has exact
valuation
\[
(m_0-1)+(m_0-r)=n-r-1.
\]
After division by $2^{n-r-1}$, every such contribution is odd.  The number
of admissible labeled pairings is
\[
(n_0-1)!!(n_1-1)!!,
\]
with the convention $(-1)!!=1$.  This number is odd.  Thus an odd number
of odd normalized contributions are being added, and no cancellation can
raise the $2$-adic valuation.  We conclude that
\begin{equation}\label{eq:walsh-per-valuation}
\vTwo\bigl(\per W_{\mathbf m}^{\xi,\psi}\bigr)
=n-r-1.
\end{equation}

Finally, since $n-2$ has binary digit sum $r-1$, Legendre's formula gives
\[
\vTwo((n-2)!)
=(n-2)-(r-1)
=n-r-1.
\]
Combining this identity, \eqref{eq:coeff-permanent}, and
\eqref{eq:walsh-per-valuation} yields
\[
\vTwo\bigl(b_{\xi,\psi}(\mathbf m)\bigr)
=
n-r-1-\vTwo(\mathbf m!)
=
\vTwo\left(\frac{(n-2)!}{\mathbf m!}\right),
\]
as required.
\end{proof}

We shall also use the following immediate character-translation identity.

\begin{lemma}[Character shift]\label{lem:character-shift}
Let $\eta\neq0$.  For every $\xi\in\whH$ and every profile
$\mathbf m$ of size $n-2$ with $S=\sigma(\mathbf m)$,
\begin{equation}\label{eq:character-shift}
b_{\xi,\xi+\eta}(\mathbf m)
=
(-1)^{\langle\xi,S\rangle}b_{0,\eta}(\mathbf m).
\end{equation}
\end{lemma}

\begin{proof}
Translation of the character index by $\xi$ gives
\[
\Theta_{\xi,\xi+\eta}(z)
=
\Theta_{0,\eta}\bigl((\eps_\xi(x)z_x)_{x\in H}\bigr).
\]
Taking the coefficient of $z^{\mathbf m}$ gives
\[
\prod_{x\in H}\eps_\xi(x)^{m_x}
=
(-1)^{\langle\xi,S\rangle},
\]
which proves \eqref{eq:character-shift}.
\end{proof}

The deleted coefficient lemma controls the character side, but the prescribed
internal directions $p$ and $q$ have not yet entered.  They enter through the
two Walsh minors attached to the deleted vertex pairs.  The next calculation
is the boundary filter: a character difference $\eta$ survives a pair of
vertices at internal difference $u$ exactly when
$\langle\eta,u\rangle=1$.

\begin{lemma}[A two-by-two Walsh minor]\label{lem:two-by-two}
Let $\eta\neq0$, let $Q=\{\xi,\xi+\eta\}$, and let
\[
A=\{a,a+u\},
\qquad
u\neq0.
\]
Then
\[
\det W_{Q,A}=0
\quad\Longleftrightarrow\quad
\langle\eta,u\rangle=0.
\]
If the rows are ordered as $(\xi,\xi+\eta)$ and the columns as
$(a,a+u)$, then for $\langle\eta,u\rangle=1$,
\begin{equation}\label{eq:2by2-walsh}
\det W_{Q,A}
=
-2(-1)^{\langle\xi,u\rangle+\langle\eta,a\rangle}.
\end{equation}
\end{lemma}

\begin{proof}
Direct expansion gives
\begin{align*}
\det W_{Q,A}
&=
\eps_\xi(a)\eps_{\xi+\eta}(a+u)
-
\eps_\xi(a+u)\eps_{\xi+\eta}(a)\\
&=
\eps_\xi(u)\eps_\eta(a)
\bigl(\eps_\eta(u)-1\bigr).
\end{align*}
This is zero when $\eps_\eta(u)=1$ and equals
$-2\eps_\xi(u)\eps_\eta(a)$ otherwise.
\end{proof}

We now combine the previous lemmas into the exact Fourier formula needed
for the two-hole argument.

\subsection{From deleted coefficients to the two-hole matching}
\label{subsec:two-hole-transform}

The deleted coefficient still depends on the unknown locations of the two
internal edges.  The next formula packages this dependence as a Walsh
transform.  Its invertibility is the final global step of the proof.

\begin{proposition}[Deleted-circulant coefficient as a Walsh transform]
\label{prop:minor-walsh-transform}
Let $p,q\in H\setminus\{0\}$ and let $\mathbf m$ be a profile
satisfying
\[
|\mathbf m|=n-2,
\qquad
\sigma(\mathbf m)=p+q.
\]
For $a,b\in H$, put
\[
A_a=\{a,a+p\},
\qquad
B_b=\{b,b+q\},
\]
and for $\eta\neq0$ put
\[
b_\eta(\mathbf m)=b_{0,\eta}(\mathbf m)
=[z^{\mathbf m}]\Theta_{0,\eta}(z).
\]
Then there exists a sign $\sigma_{a,b}\in\{\pm1\}$ such that
\begin{equation}\label{eq:minor-walsh-transform}
[z^{\mathbf m}]\det C(z)_{A_a^c,B_b^c}
=
\sigma_{a,b}\frac{2}{n}
\sum_{\substack{\eta\in\whH\\
\langle\eta,p\rangle=1\\
\langle\eta,q\rangle=1}}
(-1)^{\langle\eta,a+b\rangle}b_\eta(\mathbf m).
\end{equation}
\end{proposition}

\begin{proof}
Take the coefficient of $z^{\mathbf m}$ in
\eqref{eq:general-complementary-minor}.  Group the unordered character pairs
$Q=\{\xi,\psi\}$ by their nonzero difference
\[
\eta=\xi+\psi.
\]
For fixed $\eta\neq0$, there are exactly $n/2$ unordered pairs
$\{\xi,\xi+\eta\}$.

By Lemma~\ref{lem:two-by-two}, a term can survive only if
\[
\langle\eta,p\rangle
=
\langle\eta,q\rangle
=1.
\]
For such an $\eta$, let $Q=\{\xi,\xi+\eta\}$.  Reordering the two
columns of $W_{Q,A_a}$ as $(a,a+p)$ and those of $W_{Q,B_b}$ as
$(b,b+q)$ may introduce signs depending on $a$ and $b$, but these signs are
independent of $Q$ and are absorbed into the overall factor
$\sigma_{a,b}$.  The row-order sign is the same in the two determinants and
therefore cancels in their product.  Using \eqref{eq:2by2-walsh} twice then
gives
\begin{align*}
\det W_{Q,A_a}\det W_{Q,B_b}
&=
4(-1)^{\langle\xi,p+q\rangle
+\langle\eta,a+b\rangle}.
\end{align*}
By Lemma~\ref{lem:character-shift} and the assumption
$\sigma(\mathbf m)=p+q$,
\[
b_{\xi,\xi+\eta}(\mathbf m)
=
(-1)^{\langle\xi,p+q\rangle}b_\eta(\mathbf m).
\]
Thus the two factors depending on $\xi$ cancel, and every one of the
$n/2$ character pairs with difference $\eta$ contributes
\[
4(-1)^{\langle\eta,a+b\rangle}b_\eta(\mathbf m).
\]
Multiplying by the factor $n^{-2}$ from
Lemma~\ref{lem:general-complementary-minor} yields
\[
\frac1{n^2}\cdot\frac n2\cdot4=\frac2n,
\]
which proves \eqref{eq:minor-walsh-transform}.
\end{proof}

\begin{example}[The Walsh filter in the running two-hole instance]
\label{ex:two-hole-filter}
Return to Example~\ref{ex:running-two-hole}.  Here
\[
    p=01,
    \qquad
    q=10,
    \qquad
    \sigma(\mathbf m)=00+11=11=p+q.
\]
The two minor conditions
\[
    \langle\eta,p\rangle=1,
    \qquad
    \langle\eta,q\rangle=1
\]
leave only the character difference $\eta=11$.  Thus all deleted character
pairs with the wrong direction vanish automatically, and the entire boundary
dependence is carried by one surviving Walsh frequency.  Moreover,
\[
    \langle11,\sigma(\mathbf m)\rangle
    =
    \langle11,11\rangle
    =0,
\]
so Lemma~\ref{lem:two-deleted-walsh} guarantees that the corresponding deleted
coefficient is nonzero.  In fact,
\[
    b_{11}(\mathbf m)
    =[z_{00}z_{11}]\Theta_{0,11}(z)
    =-2.
\]
For the choice $a=b=00$ used in Example~\ref{ex:running-two-hole}, the Walsh
sum in Proposition~\ref{prop:minor-walsh-transform} therefore has a single
nonzero term.  Up to the harmless complementary-minor sign, the proposition
recovers the coefficient
\[
    [z_{00}z_{11}]\det C(z)_{A^c,B^c}=1
\]
computed directly there.  The general proof is the same mechanism with many
surviving frequencies: the Walsh transform guarantees that they cannot vanish
for every boundary $a+b$.
\end{example}

\begin{proof}[Proof of Theorem~\ref{thm:two-hole}]
If $s=2$, then $n=2$ and both requests are internal.  The zero-sum condition
forces them to be equal, say $p,p$, and the two cosets of $H$ are paired
internally in direction $p$.  Hence assume $s\geq3$.

Let
\[
    p,q\in H\setminus\{0\}
\]
be the two holes, counted with multiplicity.  Write the remaining $n-2$
crossing requests uniquely as
\[
    \alpha+d_1,\ldots,\alpha+d_{n-2},
    \qquad
    d_i\in H.
\]
Let $\mathbf m$ be the multiplicity profile of
\[
    d_1,\ldots,d_{n-2}.
\]
Since $n-2$ is even and the total request sum is zero,
\[
    0
    =
    p+q+\sum_{i=1}^{n-2}(\alpha+d_i)
    =
    p+q+\sum_{i=1}^{n-2}d_i.
\]
Hence
\[
    |\mathbf m|=n-2,
    \qquad
    \sigma(\mathbf m)=p+q.
\]

For $\eta\in\whH$, define
\[
h(\eta)
=
\begin{cases}
 b_\eta(\mathbf m),
 &\langle\eta,p\rangle
 =\langle\eta,q\rangle=1,\\
 0,&\text{otherwise}.
\end{cases}
\]
The support condition is nonempty.  If $p=q$, it is a nonempty affine
hyperplane of $\whH$.  If $p\neq q$, then $p$ and $q$ are linearly
independent over $\F_2$, and the two affine equations are simultaneously
solvable.

For every $\eta$ in the support of $h$,
\[
    \langle\eta,\sigma(\mathbf m)\rangle
    =
    \langle\eta,p+q\rangle
    =
    1+1
    =
    0.
\]
Lemma~\ref{lem:two-deleted-walsh} therefore gives
\[
    b_\eta(\mathbf m)\neq0.
\]
Thus
\[
    h\not\equiv0.
\]

The Walsh transform on $H$ is invertible.  Hence there exists $t\in H$ such
that
\[
    \sum_{\eta\in\whH}
    (-1)^{\langle\eta,t\rangle}h(\eta)
    \neq0.
\]
Choose $a,b\in H$ with
\[
    a+b=t,
\]
and put
\[
    A=\{a,a+p\},
    \qquad
    B=\{b,b+q\}.
\]
By Proposition~\ref{prop:minor-walsh-transform},
\[
    [z^{\mathbf m}]
    \det C(z)_{A^c,B^c}
    \neq0.
\]
Lemma~\ref{lem:deleted-certificate} yields a bijection
\[
    \pi:H\setminus A\longrightarrow H\setminus B
\]
whose realized difference profile is $\mathbf m$.

We now construct the matching in $\F_2^s$.  Use the internal edge
\[
    \{a,a+p\}\subseteq H
\]
for the hole $p$, and the internal edge
\[
    \{\alpha+b,\alpha+b+q\}\subseteq H+\alpha
\]
for the hole $q$.  For every $x\in H\setminus A$, use the crossing edge
\[
    \{x,\alpha+\pi(x)\}.
\]
Its difference is
\[
    x+\alpha+\pi(x)
    =
    \alpha+(x+\pi(x)),
\]
so the crossing edges realize exactly
\[
    \alpha+d_1,\ldots,\alpha+d_{n-2}
\]
with multiplicity.  The internal and crossing edges are pairwise disjoint and
cover $\F_2^s$.  Hence they form a Hadamard pair solution for $M$.
\end{proof}

\begin{remark}[No induction on the hyperplane]
The proof uses $H$ only as the abelian group on which the Walsh transform is
taken.  It does not assume Conjecture~\ref{conj:bgs} in dimension $s-1$.
The deleted minor retains the exact boundary information created by the two
internal edges.
\end{remark}

\section{Alternating-Path Exchanges and Constructive Symmetric Cases}
\label{sec:alternating-exchange}
\label{sec:constructive-complements}

We now return to the constructive side of the Hall case.  The algorithm
\emph{BSolution}~\cite{YohananovYaakobi2022} solves the same all-crossing
hyperplane family as Hall, but keeps the two affine sides explicit.  Its local
operation transfers the unresolved redundancy along an alternating path.
We first formalize this exchange and its endpoint limitation and then use it in
two symmetric multiplicity families.

\subsection{The \texorpdfstring{\emph{BSolution}}{BSolution} exchange and its endpoint limitation}
\label{subsec:bsolution-exchange}

The exchange is stated for an arbitrary nonzero direction $x\in\F_2^s$,
rather than only for a unit vector.

Fix a current ordering $\cH=[h_0,h_1,\ldots,h_{2n-1}]$ and let
\[
    \cP(\cH)=\bigl\{\{h_{2t},h_{2t+1}\}:t\in[n]\bigr\}
\]
be the current perfect matching.  For $x\in\F_2^s\setminus\{0\}$,
translation by $x$ gives the perfect matching
\[
    \cT_x
    =
    \bigl\{\{h,h+x\}:h\in\F_2^s\bigr\},
\]
where each unordered edge is retained once.

\begin{definition}[The graph $G_x(\cH)$]
\label{def:Gx}
For $x\neq0$, define
\[
    G_x(\cH)=\cP(\cH)\cup\cT_x.
\]
If the same unordered pair belongs to both matchings, the two edges are
retained as parallel edges, one from each matching.
\end{definition}

Every vertex is incident with one edge of $\cP(\cH)$ and one edge of
$\cT_x$.  Hence every connected component of $G_x(\cH)$ is an even
alternating cycle.

Following the terminology of~\cite{YohananovYaakobi2022}, we call an
alternating path used by \emph{BSolution} to transfer the unresolved
redundancy from the current internal pair to another unfixed internal pair a
\emph{good-path}.  The algorithm is always referred to as \emph{BSolution};
the term \emph{good-path} refers only to the alternating path on which its
local exchange is performed.

\begin{example}[An eight-vertex alternating graph]
Let $s=3$ and let the current ordering of the Hadamard columns be
\[
\begin{aligned}
    h_0&=000, & h_1&=100, & h_2&=010, & h_3&=110,\\
    h_4&=001, & h_5&=101, & h_6&=011, & h_7&=111.
\end{aligned}
\]
Thus the current matching is
\[
    \cP(\cH)
    =
    \bigl\{
        \{h_0,h_1\},
        \{h_2,h_3\},
        \{h_4,h_5\},
        \{h_6,h_7\}
    \bigr\}.
\]
Take
\[
    x=101.
\]
Then
\[
    \cT_x
    =
    \bigl\{
        \{h_0,h_5\},
        \{h_1,h_4\},
        \{h_2,h_7\},
        \{h_3,h_6\}
    \bigr\}.
\]
The graph $G_x(\cH)$ is shown in Figure~\ref{fig:Gx-eight}.  Dashed edges are
current-pairing edges and solid edges belong to $\cT_x$.  The graph consists
of two alternating $4$-cycles.  In the upper cycle, the one-edge path
$h_0-h_5$ is a simple alternating path whose edge belongs to $\cT_x$ and whose
endpoints lie in two distinct current pairs.  Swapping $h_0$ and $h_5$ changes
exactly the sums of $\{h_0,h_1\}$ and $\{h_4,h_5\}$, adding $x$ to both.  This
is the local mechanism used below.

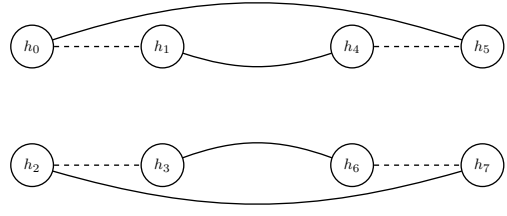
\begin{figure}[t]
\centering
\resizebox{0.50\textwidth}{!}{%
\begin{tikzpicture}[
    x=1.25cm,
    y=1.25cm,
    v/.style={circle,draw,thick,minimum size=9mm,inner sep=0pt,font=\small},
    pair/.style={thick,dashed},
    trans/.style={thick,line cap=round}
]
\node[v] (h0) at (0,2) {$h_0$};
\node[v] (h1) at (2.2,2) {$h_1$};
\node[v] (h4) at (5.4,2) {$h_4$};
\node[v] (h5) at (7.6,2) {$h_5$};

\node[v] (h2) at (0,0) {$h_2$};
\node[v] (h3) at (2.2,0) {$h_3$};
\node[v] (h6) at (5.4,0) {$h_6$};
\node[v] (h7) at (7.6,0) {$h_7$};

\draw[pair] (h0)--(h1);
\draw[pair] (h4)--(h5);
\draw[pair] (h2)--(h3);
\draw[pair] (h6)--(h7);

\draw[trans] (h0) to[bend left=18] (h5);
\draw[trans] (h1) to[bend right=18] (h4);
\draw[trans] (h3) to[bend left=20] (h6);
\draw[trans] (h2) to[bend right=16] (h7);
\end{tikzpicture}%
}
\caption{The graph $G_x(\cH)$ for the eight-vertex example.  Dashed edges are
the current pairs in $\cP(\cH)$ and solid edges are the translation edges in
$\cT_x$.  Each connected component is an alternating cycle.}
\label{fig:Gx-eight}
\end{figure}
\end{example}

The local exchange used below is the following elementary observation.

\begin{lemma}[Alternating-path switch]
\label{lem:alternating-path-switch}
Let $P$ be a simple alternating path in $G_x(\cH)$ whose first and last edges
belong to $\cT_x$, and assume that the two endpoints of $P$ belong to distinct
current pairs.  Swap the endpoints of every edge of $P$ belonging to
$\cT_x$ and apply the resulting permutation of $\F_2^s$ to the current
pairs.  Then only the two current-pair sums incident with the endpoints of
$P$ change, and each is increased by $x$.  Every current-pair sum in the
interior of $P$ is unchanged.
\end{lemma}

\begin{proof}
The $\cT_x$-edges of $P$ are pairwise disjoint, so swapping their endpoints
defines a permutation of the columns of $\cH$.  Since every component of
$G_x(\cH)$ is an alternating cycle and the endpoint current pairs are distinct,
the current-pair edge incident with either endpoint is not contained in $P$ and
its other endpoint is not a vertex of $P$.  Thus exactly one endpoint of each
endpoint current pair is translated by $x$.  By contrast, every current pair
in the interior of $P$ contributes its current-pair edge to $P$, so both of its
endpoints are translated by $x$.  Hence an interior sum satisfies
$(a+x)+(b+x)=a+b$, while each endpoint-pair sum changes from $a+b$ to
$a+b+x$.  No other current pair is touched.
\end{proof}

Fix the decomposition $\F_2^s=H\sqcup(H+\alpha)$ from
Section~\ref{sec:prelim}, and call $H$ the left side and $H+\alpha$ the right
side.  Assume every request is crossing.  This is the same Hall instance solved
globally by a bijection $H\to H$ and constructively by \emph{BSolution} on the
full ambient space $V$.  Suppose the algorithm is processing an unfixed current
pair of sum $y\in H$ and wishes to realize $v\notin H$.  Put $x=v+y$.  Then
$x\notin H$, so every edge of $\cT_x$ crosses the hyperplane cut.  The
alternating cycle of $G_x(\cH)$ containing the current internal pair contains
a second internal current-pairing edge.  Otherwise all remaining current-pair
edges on the cycle would cross the cut; together with the $\cT_x$-edges, the
closed cycle would contain an odd number of cut-crossing edges, which is
impossible.

Every previously fixed request is crossing, so the second internal pair is
unfixed.  Removing the two internal current-pairing edges from the cycle leaves
an alternating path between them whose first and last edges belong to
$\cT_x$.  Lemma~\ref{lem:alternating-path-switch} changes the initial pair sum
from $y$ to $y+x=v$ and preserves all fixed crossing pairs in the interior.
This is the local exchange underlying \emph{BSolution}.  Hall and
\emph{BSolution} therefore solve the same all-crossing problem: Hall chooses
the complete crossing matching globally, whereas \emph{BSolution} fixes the
requests one by one and transfers the unresolved redundancy from one affine side to the other.

The transfer does not control its boundary.  In the partial-solution
interpretation, after the current request is fixed, a vertex is released on the
starting side and the good-path terminates at an unfixed pair on the opposite
side.  Which vertex is released and which terminal pair is reached depend on
the path.  This is harmless in the Hall case, where no prescribed internal
boundary must survive, and it is the precise mechanism behind the terminal
obstruction described in Section~\ref{sec:prelim}.  The constructive proofs
below succeed only when multiplicity symmetry either restores a complete Hall
instance in a quotient or provides enough redundancy to stop before a
protected distance is lost.

Under even multiplicities, quotient symmetry restores enough Hall structure to
obtain explicit constructions.  We first give a quotient proof of the
symmetric two-hole case and then treat the four-hole patterns $p,p,p,p$ and
$p,p,q,q$.  Whenever \emph{BSolution} is applied to a quotient below, it acts
on the full quotient ambient space and solves an all-crossing Hall instance
there; the subsequent parallel lift is a separate step.  The $p,p,q,q$ case
additionally uses the endpoint bookkeeping above and stops the coset-level
procedure before the protected distance can be destroyed.

\subsection{Two holes under even multiplicities}
\label{sec:two-hole-even}

Theorem~\ref{thm:two-hole} already contains the case in which every request
occurs with even multiplicity.  In this subfamily, however, the Walsh-minor
argument is unnecessary.  Quotienting by the common internal direction turns
the problem into an all-crossing Hall instance on a smaller full ambient
space.  We solve that Hall instance constructively by \emph{BSolution} and then
lift its crossing edges in parallel.

\begin{theorem}[Constructive even-multiplicity two-hole case]
\label{thm:two-hole-even}
Let $M$ be a zero-sum two-hole instance with respect to $H$, and suppose that
every request vector occurs with even multiplicity in $M$.  Then a Hadamard
pair solution can be constructed by reducing to an all-crossing Hall instance
on a quotient space, solving that quotient instance by \emph{BSolution}, and
lifting the resulting matching.
\end{theorem}

\begin{proof}
Since exactly two requests lie in $H$ and every request has even multiplicity,
the two holes are equal.  Write them as
\[
    p,p,
    \qquad
    p\in H\setminus\{0\}.
\]
If $s=2$, then there are no crossing requests; pair the two vertices of $H$
and the two vertices of $H+\alpha$ in direction $p$.  Hence assume $s\geq3$.
The crossing requests also occur with even multiplicity, so we may write them
as
\[
    v_1,v_1,\;
    v_2,v_2,\;
    \ldots,\;
    v_m,v_m,
    \qquad
    v_t\in H+\alpha,
\]
where
\[
    m=\frac{n-2}{2}=\frac n2-1.
\]

Let
\[
    L=\langle p\rangle=\{0,p\}
\]
and pass to the quotient
\[
    \overline V=\F_2^s/L.
\]
Since $p\in H$, the image
\[
    \overline H=H/L
\]
is a hyperplane of $\overline V$.  Put
\[
    \overline\alpha=\alpha+L.
\]
Every
\[
    \overline v_t=v_t+L
\]
lies outside $\overline H$.

Put
\[
    \overline u
    =
    \sum_{t=1}^{m}\overline v_t.
\]
The number $m=2^{s-2}-1$ is odd, and every $\overline v_t$ is crossing with
respect to $\overline H$.  Hence $\overline u$ is also crossing and in
particular nonzero.  The quotient request list
\[
    \overline v_1,\ldots,\overline v_m,\overline u
\]
has
\[
    m+1=\frac n2=2^{s-2}
\]
requests and total sum zero.  All of them lie outside the hyperplane
$\overline H$.

Apply the constructive \emph{BSolution} algorithm to this quotient instance.
It produces a perfect matching of $\overline V$ realizing
\[
    \overline v_1,\ldots,\overline v_m,\overline u.
\]
Delete the quotient edge realizing $\overline u$.  Because $\overline u$ is
crossing, the two unmatched quotient vertices are one coset
\[
    X_0\in\overline H
\]
and one coset
\[
    Y_0\in\overline H+\overline\alpha.
\]

Consider any remaining quotient edge
\[
    \{X,Y\}
\]
realizing $\overline v_t$.  Choose $x\in X$.  Since
\[
    X+Y=\overline v_t,
\]
the coset $Y$ is
\[
    x+v_t+L.
\]
Thus the two parallel edges
\[
    \{x,x+v_t\},
    \qquad
    \{x+p,x+p+v_t\}
\]
lift the quotient edge and realize two copies of $v_t$.  Lifting every
remaining quotient edge in this way realizes all crossing requests and uses
every vector except the two $p$-cosets corresponding to $X_0$ and $Y_0$.

Write these two unused cosets as
\[
    X_0=\{x_0,x_0+p\},
    \qquad
    Y_0=\{y_0,y_0+p\}.
\]
Pair internally inside each coset:
\[
    \{x_0,x_0+p\},
    \qquad
    \{y_0,y_0+p\}.
\]
Both edges have difference $p$, and therefore realize the two holes.  Together
with the lifted quotient matching, they give a Hadamard pair solution for
$M$.
\end{proof}

\begin{remark}[Algorithmic content]
The proof is not merely existential.  The role of \emph{BSolution} is exactly
the same as in Hall: it acts on the full quotient ambient space
$\overline V$ and constructs the all-crossing quotient matching.  The passage
from a quotient edge to two parallel edges of $V$ is the subsequent lifting
step and is not part of \emph{BSolution} itself.
\end{remark}

\subsection{Four holes under even multiplicities}
\label{sec:four-hole-even}

Suppose exactly four requests lie in $H$ and every request vector in the
full instance occurs with even multiplicity.  The four internal requests are
then either four copies of one vector or two copies each of two distinct
vectors.  In this symmetric subfamily, the quotient and alternating-path
mechanisms give a direct constructive solution.

\begin{theorem}[Constructive even-multiplicity four-hole case]
\label{thm:four-hole-even}
Let $M$ be a zero-sum four-hole instance with respect to $H$, and suppose that
every request vector occurs with even multiplicity in $M$.  Then $M$ has a
Hadamard pair solution that can be constructed explicitly.
\end{theorem}

\begin{proof}
Because the four holes have even multiplicities, there are two cases.

\medskip
\noindent\textbf{Case 1: the four holes are $p,p,p,p$.}
Put
\[
    L=\langle p\rangle=\{0,p\}
\]
and pass to the quotient
\[
    \overline V=\F_2^s/L.
\]
Group the crossing requests into identical pairs
\[
    v_1,v_1,\;\ldots,\;v_m,v_m.
\]
Each pair descends to one crossing quotient request
\[
    \overline v_t=v_t+L.
\]
Let
\[
    S=\sum_{t=1}^{m}\overline v_t.
\]
Since the number of quotient requests is even, $S$ lies in the quotient
hyperplane $\overline H=H/L$.  Choose any crossing vector $w\notin\overline H$
and append the two crossing quotient requests
\[
    w,
    \qquad
    w+S.
\]
The completed quotient list is zero-sum and all its requests are crossing.
Apply \emph{BSolution} in $\overline V$.  Delete the two quotient edges
assigned to the auxiliary requests and lift every remaining quotient edge in
parallel: if
\[
    X=\{x,x+p\}
\]
is paired with $X+\overline v_t$, use
\[
    \{x,x+v_t\},
    \qquad
    \{x+p,x+p+v_t\}.
\]
These lifted edges realize all crossing requests.  The two deleted quotient
edges leave four unused $p$-cosets.  Pair internally inside each of them; the
four resulting edges all have difference $p$ and realize the four holes.

\medskip
\noindent\textbf{Case 2: the holes are $p,p,q,q$ with $p\neq q$.}
If $s=3$, then $n=4$ and there are no crossing requests.  Pair $H$ by the two
edges of direction $p$ and pair $H+\alpha$ by the two edges of direction $q$.
Hence assume $s\geq4$.

Write the crossing requests as
\[
    v_1,v_1,\;
    v_2,v_2,\;
    \ldots,\;
    v_m,v_m,
    \qquad
    v_t\in H+\alpha,
\]
where $2m+4=n$ and $m=n/2-2\geq1$.  Put
$L=\langle p\rangle=\{0,p\}$ and let
$\overline V=\F_2^s/L$, $\cK_0=H/L$, and
$\cK_1=(H+\alpha)/L$.  Write $\overline v_t=v_t+L$ and
$\overline q=q+L$.  Every $\overline v_t$ is crossing with respect to the
quotient hyperplane $\cK_0$, whereas $\overline q$ is a nonzero internal
direction.  A quotient edge realizing $\overline v_t$ lifts to the two
parallel point-level edges realizing $v_t,v_t$.

We first fix the labeled quotient requests
$\overline v_1,\ldots,\overline v_{m-1}$ while protecting two left quotient
vertices at difference $\overline q$.  Choose $A\in\cK_0$ and protect the
distinct vertices $A$ and $A+\overline q$.  Suppose $j<m-1$ quotient requests
have already been fixed.  Exactly $k=n/2-j$ unfixed quotient vertices remain
on each side, and throughout this stage $k\geq4$.  Re-pair the unfixed
vertices arbitrarily, without changing any fixed pair, so that
$\{A,A+\overline q\}$ is a current pair and there is another current internal
pair $E$ on the left disjoint from it.  Such a temporary perfect matching
always exists.  If $k$ is even, pair the remaining left vertices and all right
vertices internally.  If $k$ is odd, then $k\geq5$; after choosing the
protected pair and $E$, use one temporary crossing pair and pair all remaining
vertices internally.

Let $y$ be the current difference of $E$ and, for the next request
$\overline v_{j+1}$, put $x=y+\overline v_{j+1}$.  Since $y$ is internal and
$\overline v_{j+1}$ is crossing, $x$ is crossing.  Traverse the alternating
cycle of $G_x$ from $E$ starting with a translation edge.  Every translation
edge crosses the hyperplane cut.  Consequently, until an internal current pair
is met, each current crossing pair returns the traversal to the left and the
next translation edge moves it back to the right.  Thus the first internal
current pair encountered after leaving $E$ lies on the right.  Such a pair
exists by the same cut-parity argument as in the Hall case.  It is unfixed,
because every fixed quotient request is crossing.  The path from $E$ to this
first right internal pair has distinct endpoint current pairs and avoids the
protected left pair.  Apply Lemma~\ref{lem:alternating-path-switch}.  The sum
of $E$ becomes $y+x=\overline v_{j+1}$; declare this endpoint pair fixed.  All
previously fixed pairs are preserved, the other endpoint remains unfixed, and
the protected pair is untouched.  This closes the induction.

After the first $m-1$ quotient requests have been fixed, exactly three
$p$-cosets remain unused on each side.  On the left they are the two protected
cosets $A,A+\overline q$ and one additional coset $R$, whose position is not
controlled.  On the right denote the three unused $p$-cosets by
$B_1,B_2,B_3$.  Lifting the fixed quotient edges in parallel realizes
$v_1,v_1,\ldots,v_{m-1},v_{m-1}$.

This is the point at which one must stop the coset-level algorithm.  A further
coset-level \emph{BSolution} step for the remaining pair $v_m,v_m$ would require two
left redundant cosets.  Since only $R$ is unprotected, one of
$A,A+\overline q$ would necessarily be used.  Suppose, for example, that the
step starts from $R$ and $A$.  The good-path does not return $A$ at its end;
it releases an uncontrolled coset $A'$ on the left.  The remaining protected
candidate is then $A',A+\overline q$, and there is no reason for these two
cosets to remain at difference $\overline q$.  The distance that was preserved
throughout the coset-level stage would be lost.

Instead, leave the quotient level and perform one ordinary point-level
\emph{BSolution} step for only one copy of $v_m$.  Write $R=\{r,r+p\}$.
Temporarily pair the two points of $R$ together, pair each of the three unused
right $p$-cosets internally, and pair the four points of
$A\cup(A+\overline q)$ by the two parallel $q$-edges.  These temporary pairs
do not alter any previously fixed crossing pair.  Start the point-level
\emph{BSolution} step from the internal pair $R$ and stop at the first internal
pair reached on the right.  The same cut-parity argument as above gives such a
good-path and ensures that the two protected $q$-edges on the left are not
touched.  The step fixes one copy of $v_m$ and preserves all previously fixed
crossing requests.  Its exact right endpoint is not controlled, and no such
control is needed here.

The endpoint exchange can disturb at most one of the three right $p$-cosets,
so two complete right $p$-cosets remain unused; pairing inside them realizes
$p,p$.  The two protected left pairs already realize $q,q$.  Writing
$A=\{a,a+p\}$, these are $\{a,a+q\}$ and
$\{a+p,a+p+q\}$.

After these four internal pairs are added, exactly one unused point remains on
the left and one on the right.  All requests have been realized except for the
second copy of $v_m$.  Since the sum of all vertices of $V$ is zero and the
total request sum is zero, the sum of the final two unused vertices must be
$v_m$.  Their pair therefore realizes the second copy of $v_m$.  Together
with the lifted quotient matching and the point-level \emph{BSolution} step, this
gives a Hadamard pair solution for $M$.
\end{proof}

\begin{remark}[Why the $p,p,q,q$ argument succeeds]
The proof does not gain control over the endpoint of the good-path.  Instead, it
avoids needing that control.  Two left $p$-cosets at difference
$\overline q$ are protected until the coset-level stage reaches the $3+3$
configuration.  The last coset-level step is deliberately not performed,
because it would necessarily destroy that distance.  A point-level \emph{BSolution}
step is then run from the third, uncontrolled left coset and is used for only
one copy of $v_m$.  Two complete right $p$-cosets survive, while the second
copy of $v_m$ is forced by the global sum.  This is precisely the extra
symmetry absent from the general two-hole problem.
\end{remark}

\ifshowopenproblems
\else
\section{Conclusion}
\label{sec:beyond-four}

The two-hole theorem isolates the first nontrivial failure of the complete
Hall permutation.  With all requests crossing, the problem is encoded by the
full group circulant.  Two internal requests remove two vertices from each
affine side, so the crossing core is instead encoded by an unknown
complementary minor of the same matrix.  The main point of the present paper
is that this boundary defect can still be detected globally.

The character method first gives an exact noncancellation theorem for the
binary Hall case: every zero-sum determinant coefficient has the same
$2$-adic valuation as the corresponding multinomial coefficient.  Passing to
two deleted character factors then produces the compatibility filter needed
for the two-hole boundary.  The final Walsh transform does not construct the
deleted vertices locally; rather, its invertibility forces the existence of a
boundary choice for which the required minor coefficient is nonzero.

This distinction also clarifies the role of the constructive arguments.
\emph{BSolution} keeps the two affine sides explicit and is particularly
natural for the all-crossing case, but its local exchange does not control the
terminal free boundary.  The general two-hole theorem therefore uses the
global character-minor argument.  Under additional even-multiplicity symmetry,
quotient structure restores enough control to obtain direct constructive
solutions.

Thus the two-hole problem is not merely Hall with two exceptional requests.
It is the first case in which the complete permutation is punctured at unknown
locations and the geometry of the missing boundary becomes part of the
matching problem itself.  The complementary-minor viewpoint is designed
precisely for this first punctured Hall structure, and the theorem above gives
a complete solution without a multiplicity assumption.
\fi

\end{document}